\documentclass[10pt,a4paper]{amsart}

\usepackage[left=3.3cm,right=3.3cm,top=4cm,bottom=3.5cm]{geometry}

\usepackage[T1]{fontenc}
\usepackage{amssymb}
\usepackage{thmtools}
\usepackage{enumitem}
\usepackage{mathrsfs}
\usepackage{accents}
\usepackage{lastpage}       
\usepackage{amsmath}
\usepackage{amssymb}

\newcommand{\bb}{\mathbb}
\newcommand{\cal}{\mathcal}

\renewcommand{\th}{\textsuperscript{th}\ }
\providecommand{\eps}{\varepsilon}

\newcommand{\norm}[1]{\left\lVert#1\right\rVert}
\newcommand{\seminorm}[1]{\left[#1\right]}
\providecommand{\tensor}{\otimes}

\AtBeginDocument{
  \let\div\relax
  \let\d\relax
  \DeclareMathOperator{\div}{div}
  \newcommand{\d}{\ensuremath{\mathrm{d}}}
}

\newcommand{\pdeproblem}[6]{
  \setlength{\arraycolsep}{0pt}
  \renewcommand{\arraystretch}{1.2}
  \left\{\begin{array}{r @{\ } l @{\ } l}
    #1 &= #2 &\text{ in } #3 \\
    #4 &= #5 &\text{ on } #6 \\
  \end{array}\right.
}

\DeclareMathOperator{\tr}{tr}

\DeclareMathOperator{\dist}{dist}

\DeclareMathOperator{\sym}{Sym}
\DeclareMathOperator{\loc}{loc}

\DeclareMathOperator{\BMO}{BMO}
\DeclareMathOperator{\VMO}{VMO}

\newcommand{\twopartdef}[4] { \left\{ \begin{array}{ll} #1 &  #2 \\ #3 &  #4 \end{array} \right.}

\def\Xint#1{\,\mathchoice
  {\XXint\displaystyle\textstyle{#1}}%
  {\XXint\textstyle\scriptstyle{#1}}%
  {\XXint\scriptstyle\scriptscriptstyle{#1}}%
  {\XXint\scriptscriptstyle\scriptscriptstyle{#1}}%
  \!\int}
\def\XXint#1#2#3{\setbox0=\hbox{$#1{#2#3}{\int}$}\vcenter{\hbox{$#2#3$}}\kern-.5\wd0}

\def\dashint{\Xint-}

\theoremstyle{definition}
\newtheorem{thm}{Theorem}[section]
\newtheorem*{thm*}{Theorem}
\newtheorem{cor}[thm]{Corollary}
\newtheorem{defn}[thm]{Definition}

\newtheorem{lem}[thm]{Lemma}

\newtheorem{prop}[thm]{Proposition}

\newtheorem{hyp}[thm]{Hypotheses}

\newtheorem{rem}[thm]{Remark}

\usepackage{tikz}
\usetikzlibrary{backgrounds}
\usetikzlibrary{arrows,arrows.meta}
\usetikzlibrary{shapes,shapes.geometric,shapes.misc}

\tikzstyle{tikzfig}=[baseline=-0.25em,scale=0.5]

\pgfkeys{/tikz/tikzit fill/.initial=0}
\pgfkeys{/tikz/tikzit draw/.initial=0}
\pgfkeys{/tikz/tikzit shape/.initial=0}
\pgfkeys{/tikz/tikzit category/.initial=0}

\pgfdeclarelayer{edgelayer}
\pgfdeclarelayer{nodelayer}
\pgfsetlayers{background,edgelayer,nodelayer,main}

\tikzstyle{none}=[inner sep=0mm]

\newcommand{\tikzfig}[1]{%
{\tikzstyle{every picture}=[tikzfig]
\IfFileExists{#1.tikz}
  {\input{#1.tikz}}
  {%
    \IfFileExists{./figures/#1.tikz}
      {\input{./figures/#1.tikz}}
      {\tikz[baseline=-0.5em]{\node[draw=red,font=\color{red},fill=red!10!white] {\textit{#1}};}}%
  }}%
}

\tikzstyle{every loop}=[]\usepackage[citecolor=blue,colorlinks=true,bookmarksdepth=2]{hyperref}

\numberwithin{equation}{section}
\makeatletter
\@namedef{subjclassname@2020}{\textup{2020} Mathematics Subject Classification}
\makeatother

\usepackage{import}
\usepackage{xifthen}
\usepackage{pdfpages}
\usepackage{transparent}

\usepackage{doi}

\begin{document}


\title[BMO $\eps$-regularity results for elliptic systems]{BMO $\varepsilon$-regularity results for solutions to Legendre-Hadamard elliptic systems}

\author[C.~Irving]{Christopher Irving}
\thanks{\textit{Funding}: The author was supported by the EPSRC [EP/L015811/1]}
\address{Faculty of Mathematics, TU Dortmund University, Vogelpothsweg 87, 44227 Dortmund, Germany}
\email{christopher.irving@tu-dortmund.de}
\date{\today}

\begin{abstract}
  We will establish an $\eps$-regularity result for weak solutions to Legendre-Hadamard elliptic systems, under the a-priori assumption that the gradient $\nabla u$ is small in $\mathrm{BMO}.$ Focusing on the case of Euler-Lagrange systems to simplify the exposition, regularity results will be obtained up to the boundary, and global consequences will be explored. Extensions to general quasilinear elliptic systems and higher-order integrands is also discussed.
\end{abstract}

\maketitle

\tableofcontents

\section{Introduction}

In this paper we study the regularity of weak solutions to the Euler-Lagrange system
\begin{equation}\label{eq:autonomous_EL}
  - \div F'(\nabla u) = 0
\end{equation}
in $\Omega \subset \bb R^n$ where $u \colon \Omega \to \bb R^N$ is a vector-valued mapping, that is $u$ satisfies
\begin{equation}
  \int_{\Omega} F'(\nabla u) : \nabla \varphi \,\d x = 0
\end{equation}
for all $\varphi \in C^{\infty}_c(\Omega,\bb R^N).$ Henceforth referred to as \emph{$F$-extremals}, solutions to (\ref{eq:autonomous_EL}) are critical points to the functional
\begin{equation}\label{eq:F_energy}
  \cal F(w) = \int_{\Omega} F(\nabla w(x)) \,\d x.
\end{equation}
There is a considerable literature studying the partial regularity theory for minimisers of such functionals, under a suitably strict version of the \emph{quasiconvexity condition} introduced by \textsc{Morrey} \cite{article:Morrey52}. A striking feature of the vectorial ($n,N\geq 2$) setting is that minimisers need not be everywhere regular (see for instance \cite{article:degiorgi68, article:Mazya69,  article:Necas77, article:MooneySavin16}), so the best we can hope for are \emph{partial regularity} results. In the quasiconvex setting the first result in this direction was due to \textsc{Evans} in \cite{article:Evans86} which has been extended considerably since; we refer the interested reader to the monograph of \textsc{Giusti} \cite{book:Giusti03} and the references therein.

For arbitrary weak solutions of the above equation however, the work of \textsc{M\"uller \& \v{S}ver\'ak} \cite{article:MullerSverak03} shows that we cannot hope for improved regularity results. Developing the theory of convex integration for Lipschitz mappings they constructed highly irregular solutions to (\ref{eq:autonomous_EL}), including Lipschitz solutions that fail to be $C^1$ in any open subset and compactly supported solutions whose gradient is $L^q$-integrable if and only if $q \leq 2.$ These results have been extended by \textsc{Kristensen \& Taheri} \cite{article:KristensenTaheri03} for weak local minimisers, and by \textsc{Sz\'ekelyhidi} \cite{article:Szekelyhidi04} for strongly polyconvex integrands.

However it is well-known that if $u$ is suitably regular, we can infer higher regularity by a bootstrap argument. This follows for instance using the classical \emph{Schauder estimates}, where if the integrand $F$ is smooth and suitably convex, any $C^{1,\alpha}$ solution for $\alpha \in (0,1)$ can be shown to be smooth. A natural question is to ask whether this a-priori H\"older condition can be further relaxed.

This was by \textsc{Moser} in the preprint \cite{article:Moser01}, who claimed it was sufficient to assume that $u$ was Lipschitz such that $\nabla u$ lies in the space $\VMO$ of functions of \emph{vanishing mean oscillation} as introduced by \textsc{Sarason} \cite{article:Sarason75}. This condition was motivated from related regularity results for linear elliptic systems, where the work of \textsc{Chiarenza, Frasca, \& Longo} \cite{article:ChiarenzaEtAl91} established $W^{2,p}$ estimates for linear uniformly elliptic equations where the coefficient matrix $A$ was assumed to be in $\VMO.$ A similar statement was established by \textsc{Campos Cordero} \cite{article:CamposCordero17} for quasiconvex integrands through different means, noting also an inconsistency in the proof in \cite{article:Moser01}. In this paper we will extend these results, establishing regularity up to the boundary in a more general setting.

While we focus on the case of $F$-extremals to illustrate the main ideas, it turns out the arguments do not make use of the variational structure and extends to more general Legendre-Hadamard elliptic systems. We will sketch this extension in Section \ref{sec:extensions}, where higher-order equations are also considered.

\subsection{Setup and main results}

We will study the following class of integrands; we refer the reader to Section \ref{sec:notation} for the precise notational conventions.

\begin{hyp}\label{hyp:autonomous_F}
  For $n \geq 2$ and $N \geq 1,$ let $F \colon \bb R^{Nn} \to \bb R$ satisfy the following.

  \begin{enumerate}[label=(H\arabic*)]
    \setcounter{enumi}{-1}

    \item $F$ is of class $C^2.$

    \item There is $q \geq 2$ such that $F$ satisfies the natural growth condition
      \begin{equation*}
        \lvert F(z)\rvert \leq K (1+\lvert z\rvert)^q
      \end{equation*}
      for all $z \in \bb R^{Nn}.$

    \item $F''$ satisfies a strict Legendre-Hadamard condition, namely for all $z \in \bb R^{N n}$ we have
      \begin{equation*}\label{eq:legendrehadamard}
        F''(z_0)(\xi \otimes \eta) :(\xi \otimes \eta) \geq 0
      \end{equation*}
      for all $\xi \in \bb R^N$ and $\eta \in \bb R^n,$ with equality if and only if $\xi\otimes \eta = 0.$

  \end{enumerate}
\end{hyp}

A key feature of our results is that we only need to assume a strict Legendre-Hadamard condition which is closely related to rank-one convexity of $F,$ and as the construction of \textsc{\v{S}ver\'ak} \cite{article:Sverak92} illustrates rank-one convexity is strictly weaker than the quasiconvexity condition of \textsc{Morrey}. We also highlight that we do not require control in the $L^q$ scales from below, so this allows for all growth conditions of type $(1,q),$ that is
\begin{equation}
  \lvert z\rvert -1 \lesssim F(z) \lesssim \lvert z\rvert^q +1
\end{equation} 

The key ideas are contained in the following interior regularity theorem, which we will prove in Section \ref{sec:interiorregularity}. For the precise definition of $\BMO$ functions we adopt in the text we refer the reader to Section \ref{sec:BMO}.

\begin{thm}[$\BMO$ $\eps$-regularity theorem]\label{thm:autonomous_regularity}
  Suppose $F$ satisfies Hypotheses \ref{hyp:autonomous_F}. Then for all $M > 0$ and $\alpha \in (0,1),$ there is $\eps = \eps(M,F,\alpha)>0$ such that for any ball $B_R(x_0) \subset \bb R^n$ if $u$ is $F$-extremal in $B_R(x_0)$ with $\lvert (\nabla u)_{B_R(x_0)}\rvert \leq M$ and 
  \begin{equation}
     \seminorm{\nabla u}_{\BMO(B_R)} \leq \eps,
  \end{equation}
  we have $u$ is $C^{1,\alpha}$ on $\overline {B_{R/2}(x_0)}.$
\end{thm}

We will follow a similar strategy to the partial regularity theory for minimisers, which traces back to the works of \textsc{Morrey} \cite{article:Morrey68} and \textsc{Giusti \& Miranda} \cite{article:GiustiMiranda68} in the variational setting. This will involve establishing a suitable \emph{Caccioppoli inequality} and a harmonic approximation result, which are combined in a final iteration argument. For the former we will use a modification of an estimate which appeared in \textsc{Moser} \cite{article:Moser01}, and for the latter we will follow a recent approach of \textsc{Gmeineder \& Kristensen} \cite{article:GmeinederKristensen19} adapted to our setting.

We will also establish an analogous result up to the boundary, using ideas from \textsc{Kronz} \cite{article:Kronz05} and \textsc{Campos Cordero} \cite{article:CamposCordero17}. We will prove this in Section \ref{sec:boundaryregularity}, and will rely on technical results established in Section \ref{sec:technicalprelims}. Here we denote $\Omega_R(x_0) = \Omega \cap B_R(x_0).$

\begin{thm}[Boundary $\BMO$ $\eps$-regularity theorem]\label{thm:autonomous_boundary_regularity}
  Suppose $F$ satisfies Hypotheses \ref{hyp:autonomous_F}, $\Omega \subset \bb R^n$ is a $C^{1,\beta}$ domain for some $\beta \in (0,1),$ $g \in C^{1,\beta}(\overline\Omega,\bb R^N),$ and let $u \in W^{1,q}_g(\Omega,\bb R^N)$ be $F$-extremal. Then for each $\alpha \in (0,\beta)$ and $M>0$ there is $\eps = \eps(M,F,\Omega,g,\beta,\alpha)>0$ and $\widetilde R_0 = \widetilde R_0(M,F,\Omega,g,\beta,\alpha)>0$ such that if $x_0 \in \partial\Omega$ and $0<R<\widetilde R_0$ with $\lvert (\nabla u)_{\Omega_R(x_0)}\rvert \leq M$ and  
  \begin{equation}
    \seminorm{\nabla u}_{\BMO(\Omega_R(x_0))} \leq \eps,
  \end{equation}
  we have $u$ is $C^{1,\alpha}$ on $\overline{\Omega_{R/2}(x_0)}.$
\end{thm}

Patching these local regularity results we can infer global consequences, for which we will need some notation. Following \cite{article:Mazya10}, we define the \emph{infinitesimal mean oscillation} of $f \in \BMO(\Omega,\bb R^{Nn})$ as
\begin{equation}
  \{f\}_{\mathrm{osc}(\Omega)} = \limsup_{R \to 0} \sup_{\substack{B_r(x) \subset \Omega \\0<r<R}} \dashint_{B_r(x)} \lvert f-(f)_{B_r(x)}\rvert \,\d x.
\end{equation}
Note that $\{f\}_{\mathrm{osc}(\Omega)} = 0$ if and only if $f \in \VMO(\Omega,\bb R^{Nn}).$

\begin{cor}[Regularity of almost $\VMO$ Lipschitz solutions]\label{cor:lipschitzvmo_regularity}
  Suppose $F$ satisfies Hypotheses \ref{hyp:autonomous_F}, let $\Omega \subset \bb R^n$ be a $C^{1,\beta}$ domain for some $\beta \in (0,1)$ and $g \in C^{1,\beta}(\overline\Omega,\bb R^N).$ Then for each $M>0$ and $\alpha \in (0,\beta),$ there is $\eps = \eps(M,F,\Omega,g,\beta,\alpha)>0$ such that if $u \in W^{1,\infty}_g(\Omega,\bb R^N)$ is $F$-extremal such that $\norm{\nabla u}_{L^{\infty}(\Omega)} \leq M,$ and
  \begin{equation}
    \{\nabla u\}_{\mathrm{osc}(\Omega)} \leq \eps,
  \end{equation}
  then $u \in C^{1,\alpha}(\overline\Omega,\bb R^N).$
\end{cor}

It is unclear if the Lipschitz assumption can be removed; the infinitesimal mean oscillation assumption requires us to consider balls of arbitrarily small radius, which in turn requires a uniform bound on all averages $\lvert (\nabla u)_{\Omega_R(x_0)}\rvert$ for all $x_0 \in \overline\Omega$ and $R>0$ small. However this is equivalent to assuming $\nabla u$ is bounded by the Lebesgue differentiation theorem.

We point out that \textsc{Dolzmann, Kristensen, \& Zhang} \cite{article:DolzmannEtAl13} constructed an example of a minimiser of a quasiconvex integrand who gradient is unbounded but lies in $\BMO.$ We will later investigate whether this Lipschitz assumption can be relaxed under further assumptions, but for the moment we will record several other straightforward consequences.

\begin{cor}[Partial regularity of $\VMO$ solutions]
  Suppose $F$ satisfies Hypotheses \ref{hyp:autonomous_F}, let $\Omega \subset \bb R^n$ be a $C^{1,\beta}$ domain for some $\beta \in (0,1)$ and $g \in C^{1,\beta}(\overline\Omega,\bb R^N).$ Then if $u \in W^{1,q}_g(\Omega,\bb R^N)$ is $F$-extremal such that $\nabla u \in \VMO(\Omega;\bb R^{Nn}),$ letting
  \begin{equation}
    \cal R_{\overline\Omega} = \left\{ x \in \overline\Omega : \limsup_{R \to 0} \lvert (\nabla u)_{\Omega_R(x_0)}\rvert < \infty \right\},
  \end{equation}
  we have $\cal R_{\overline\Omega} \subset \overline\Omega$ is a relatively open subset of full measure and $u$ is $C^{1,\alpha}$ on $\cal R_{\overline\Omega}$ for all $\alpha \in (0,\beta).$
\end{cor}

We can also obtain a global regularity result if we assume $\nabla u$ is suitably small in both $L^1$ and $\BMO.$ The $L^1$ smallness condition allows us to cover $\overline\Omega$ by balls finitely many balls $B_{R_k}(x_k)$ such that each $\lvert (\nabla u)_{\Omega_{R_k}(x_k)}\rvert \leq 1 + \seminorm{\nabla g}_{L^{\infty}(\Omega)},$ on which we can apply our $\eps$-regularity result to obtain the following.

\begin{cor}[Regularity of $\BMO$-small solutions]
  Suppose $F$ satisfies Hypotheses \ref{hyp:autonomous_F}, let $\Omega \subset \bb R^n$ be a $C^{1,\beta}$ domain for some $\beta \in (0,1)$ and $g \in C^{1,\beta}(\overline\Omega,\bb R^N).$ Then for each $\alpha \in (0,\beta)$ there is $\eps = \eps(F,\Omega,g,\beta,\alpha)>0$ such that if $u \in W^{1,q}_g(\Omega,\bb R^N)$ is $F$-extremal in $\Omega$ with $\nabla u \in \BMO(\Omega,\bb R^{Nn})$ satisfying
  \begin{equation}
    \norm{\nabla u-\nabla g}_{L^1(\Omega)} + \seminorm{\nabla u}_{\BMO(\Omega)} \leq \eps,
  \end{equation}
  then $u \in C^{1,\alpha}(\overline\Omega,\bb R^N).$
\end{cor}

Finally we will show that the Lipschitz condition in Corollary \ref{cor:lipschitzvmo_regularity} can be removed it we assume the following uniformly controlled growth condition.

\begin{hyp}\label{hyp:controlled_F}
  For $n \geq 2,$ $N \geq 1$ and $p \geq 2,$ let $F \colon \bb R^{Nn} \to \bb R$ satisfy the following.

  \begin{enumerate}[label=(\~H\arabic*)]
    \setcounter{enumi}{-1}

    \item $F$ is of class $C^2.$

    \item We have $(1+\lvert z\rvert^2)^{-(p-2)}F''(z)$ is bounded and uniformly continuous on $\bb R^{Nn}.$\label{item:controlled_growth}

    \item For all $z \in \bb R^{N n}$ we have
      \begin{equation*}\label{eq:controlled_legendrehadamard}
        F''(z)(\xi \otimes \eta) \cdot (\xi \otimes \eta) \geq \lambda (1+\lvert z\rvert)^{p-2} \lvert \xi \rvert^2 \lvert\eta\rvert^2
      \end{equation*}
      for all $\xi \in \bb R^N$ and $\eta \in \bb R^n.$

  \end{enumerate}
\end{hyp}

Note the growth and continuity hypothesis \ref{item:controlled_growth} is satisfied if $F$ is a polynomial; in particular this includes the example of \textsc{\v{S}ver\'ak} \cite{article:Sverak92}.

\begin{thm}[$\BMO$ $\eps$-regularity in the uniformly elliptic case]\label{thm:controlled_almostvmo}
  Suppose $F \colon \bb R^{Nn} \to \bb R$ satisfies Hypotheses \ref{hyp:controlled_F}, let $\Omega \subset \bb R^N$ be a bounded $C^{1,\beta}$ domain, and let $g \in C^{1,\beta}(\overline\Omega,\bb R^N)$ for some $\beta \in (0,1).$
  Then for each $\alpha \in (0,\beta)$ there is $\eps=\eps(M,F,\Omega,g,\beta,\alpha)>0$ such that if $u \in W^{1,p}_g(\Omega;\bb R^N)$ is $F$-extremal in $\Omega$ and
  \begin{equation}
    \{\nabla u\}_{\mathrm{osc}(\Omega)} \leq \eps,
  \end{equation}
  then $u \in C^{1,\alpha}(\overline\Omega,\bb R^N).$
\end{thm}

\subsection{Connection to minimisers and quasiconvexity}

In the context of strictly quasiconvex integrands, there is a close connection between sufficiency results (whether extremals are minimising) and regularity of the extremal. One of the early results in the quasiconvex setting is due to \textsc{Zhang} \cite{article:Zhang92}, who showed that a $C^2$ extremal is absolutely minimising on small balls $B \subset \Omega.$ In the opposite direction, it was shown by \textsc{Kristensen \& Taheri} in \cite[Theorem 4.1]{article:KristensenTaheri03} that if $u \in W^{1,p} \cap W^{1,q}_{\loc}$ is a $W^{1,q}$-local minimiser for some $1 \leq q \leq \infty,$ then we can establish a partial regularity theorem (we refer the reader to the aforementioned paper for the precise terminology and results).

It was moreover established in \cite[Theorems 6.1, 7.1]{article:KristensenTaheri03} that if $u$ is a Lipschitz extremal with strictly positive second variation (it is a \emph{weak local minimiser}) then it is minimising among perturbations such that $\seminorm{\nabla\varphi}_{\BMO(\Omega)}$ is small, but that this is too weak to infer improved regularity though counterexamples. The former statement uses a modular version of the Fefferman-Stein inequality which we also use (see Section \ref{sec:BMO}), and the latter follows by adapting the construction of \textsc{M\"uller \& \v{S}ver\'ak} \cite{article:MullerSverak03}. For Lipschitz weak local minimisers however, it was shown by \textsc{Campos Cordero} \cite{article:CamposCordero17} that we can infer global regularity if we additionally assume that $\nabla u \in \VMO(\Omega).$ The proof loosely follows the compensated compactness argument used in \cite[Section 4]{article:KristensenTaheri03}.

We see that Corollary \ref{cor:lipschitzvmo_regularity} generalises the above result in \cite{article:CamposCordero17}, by removing the condition on the second variation and allowing $F$ to merely satisfy a strict Legendre-Hadamard condition \ref{eq:legendrehadamard}. Here the Legendre-Hadamard condition can be seen to be a natural relaxation in the following sense; it is proved by \textsc{Kristensen} \cite{article:Kristensen99} that \ref{eq:legendrehadamard} implies that $F$ is \emph{locally quasiconvex} in the sense that for each $z_0 \in \bb R^{Nn}$ there exists a quasiconvex function $G$ such that $F=G$ in a neighbourhood of $z_0.$ Our argument, which builds upon ideas of \textsc{Moser} \cite{article:Moser01}, streamlines this process by establishing regularity directly. In particular we note that the same Fefferman-Stein estimate used for the $\BMO$-sufficiency result in \cite{article:KristensenTaheri03} is crucially used to obtain a Caccioppoli-type inequality in \cite{article:Moser01} and Section \ref{sec:interior_caccioppoli}.

\subsection{Basic notation}\label{sec:notation}

We will briefly fix some notation that will be used throughout the text. We will equip $\bb R^n$ with the Lebesgue measure $\cal L^n,$ and if $A \subset \bb R^n$ is non-empty and open such that $\cal L^n(A)<\infty,$ for any $f \in L^1(A,\bb R^k)$ with $k\geq 1$ we define
\begin{equation}
  (f)_A := \dashint_A f \,\d x := \frac1{\cal L^n(A)} \int_A f \,\d x.
\end{equation}
We also denote by a $B_R(x_0)$ the open ball in $\bb R^n$ centred at $x_0$ with radius $R,$ and for $\Omega \subset \bb R^n$ open write $\Omega_R(x_0) = \Omega \cap B_R(x_0).$ We may write $B_R,$ $\Omega_R$ respectively if the centre point $x_0$ is clear from context.

We will denote by $\bb R^{Nn}$ the space of $N \times n$ real matrices, which we equip with the inner product $z : w = \tr(z^tw)$ and $\ell^2$-norm $\lvert z\rvert^2 = z : z$ for $z,w \in \bb R^{Nn}.$ For a differentiable map $F \colon \bb R^{Nn} \to \bb R$ we define its derivative $F' \colon \bb R^{Nn} \to \bb R^{Nn}$ as 
\begin{equation}
  F'(z)w = \left.\frac{\d}{\d t}\right\rvert_{t=0} F(z+tw),
\end{equation}
and for a differentiable map $A \colon \bb R^{Nn} \to \bb R^{Nn}$ its derivative $A'(z)$ will be a linear map $\bb R^{Nn} \to \bb R^{Nn}$ at each $z \in \bb R^{Nn},$ defined by
\begin{equation}
  A'(z)w = \left.\frac{\d}{\d t}\right\rvert_{t=0} A(z+tw).
\end{equation}
If $F$ is $C^2$ this allows us to define $F'',$ which satisfies $F''(z)v : w = F''(z)w : v$ for all $z, v, w \in \bb R^{Nn}.$

Additionally $C$ will denote a constant that may change from line to line, and if not specified in proofs they will depend only on the parameters the resulting estimate depends on.

\section{Interior regularity for \texorpdfstring{$F$}{F}-extremals}\label{sec:interiorregularity}

We begin by considering the interior regularity theory for solutions to the Euler-Lagrange system. While the techniques extend to the general case, we will present a detailed proof in this simplified setting first to illustrate the key ideas. We will refer to Section \ref{sec:technicalprelims} for some auxiliary results, but since we only apply them on balls $B$ they can be obtained through simpler means.

\subsection{Estimates for \texorpdfstring{$F$}{F}}\label{sec:autonomous_integrand}

We will consider $F \colon \bb R^{Nn} \to \bb R$ satisfying Hypotheses \ref{hyp:autonomous_F}, and fix $M>0.$ Since $F''(z)$ is uniformly continuous on compact subsets, there is $\Lambda_M>0$ and a modulus of continuity function $\omega_M \colon [0,\infty) \to [0,1]$ such that
\begin{align}\label{eq:autonomous_secondF_estimates}
  \lvert F''(z)\rvert &\leq \Lambda_M, \\
  \lvert F''(z)-F''(w)\rvert &\leq \Lambda_M \omega_M(\lvert z-w\rvert)
\end{align}
for all $z,w \in \bb R^{Nn}$ with $\lvert z\rvert,\lvert w\rvert \leq M+1.$ Here $\omega_M$ can be chosen to be a non-decreasing, continuous, and concave function such that $\omega_M(0)=0.$ Also since the strict Legendre-Hadamard condition holds uniformly on compact subsets, there is $\lambda_M>0$ such that for all $z \in \bb R^{Nn}$ with $\lvert z\rvert \leq M$ we have
\begin{equation}\label{eq:autonomous_quantitative_LH}
  F''(z)(\xi \tensor \eta) : (\xi \tensor \eta) \geq \lambda_M \lvert \xi\rvert^2\lvert \eta\rvert^2
\end{equation}
for all $\xi \in \bb R^N$ and $\eta \in \bb R^n.$ Now for $w \in \bb R^{Nn}$ with $\lvert w\rvert\leq M,$ following \textsc{Acerbi \& Fusco} \cite{article:AcerbiFusco87} consider the shifted integrand
\begin{equation}\label{eq:autonomous_shifted}
  F_w(z) = F(z+w) - F(w) - F'(w)z.
\end{equation}
Since $F''$ satisfies a Legendre-Hadamard condition, we infer $F$ is rank-one convex and so its derivative satisfies $\lvert F'(z)\rvert \leq C(n,N)K(1+\lvert z\rvert)^{q-1}.$ Hence $F_w$ satisfies the growth conditions
\begin{align}
  \lvert F_w(z)\rvert &\leq K_M( \lvert z\rvert^2 + \lvert z\rvert^q), \label{eq:autonomous_shifted_estimate1}\\
  \lvert F'_w(z)\rvert &\leq K_M( \lvert z\rvert + \lvert z\rvert^{q-1})\label{eq:autonomous_shifted_estimate2}
\end{align}
where
\begin{equation}\label{autonomous_KM}
  K_M = \Lambda_M + C(N,n)K,
\end{equation}
using the mean value theorem and distinguishing between the cases when $\lvert z\rvert \leq 1$ and $\lvert z\rvert>1.$ A similar argument gives the comparison estimate
\begin{equation}\label{eq:autonomous_F_pertubation}
  \lvert F''_w(0)z - F'_w(z)\rvert \leq K_M\,\omega_M(\lvert z\rvert)(\lvert z\rvert+\lvert z\rvert^{q-1}).
\end{equation}

\subsection{Caccioppoli-type inequality}\label{sec:interior_caccioppoli}

We now prove the following weakening of the \emph{Caccioppoli inequality of the second kind} introduced by \textsc{Evans} in \cite{article:Evans86}, which is a staple for many partial regularity proofs in the quasiconvex setting. The following estimate was essentially proved by \textsc{Moser} in \cite{article:Moser01}, and involves applying the modular version of the estimate of \textsc{Fefferman \& Stein} \cite{article:FeffermanStein72} established in Section \ref{sec:BMO} (see also Remark \ref{rem:interior_feffermanstein} at the end of this subsection).

\begin{lem}[Caccioppoli-type inequality]\label{lem:autonomous_caccioppoli}
  Suppose $F$ satisfies Hypotheses \ref{hyp:autonomous_F}, and let $M \geq 1.$ Then if $u$ is $F$-extremal in some ball $B_R(x_0) \subset \Omega$ such that $\nabla u \in \BMO(B_R,\bb R^{Nn})$ with $\seminorm{\nabla u}_{\BMO(B_R)} \leq 1$ and $\lvert(\nabla u)_{B_R}\rvert \leq M,$ then setting
  \begin{equation}\label{eq:u_affine_approximant}
    a_R(x) = (u)_{B_R} + \left( \frac{n+2}{R^2} \dashint_{B_R} u(y) \tensor (y-x_0) \,\d y \right)  \cdot (x-x_0),
  \end{equation}
  there is $\widetilde M = C(n)M$ and $C=C(n,N,q,K_{\widetilde M}/\lambda_{\widetilde M})>0$ such that
  \begin{equation}
    \dashint_{B_{R/2}} \lvert \nabla u - \nabla a_R\rvert^2 \,\d x \leq C\,\gamma\left(\seminorm{\nabla u}_{\BMO(B_R)}\right) \dashint_{B_R} \lvert \nabla u - \nabla a_R\rvert^2 \,\d x + \frac{C}{R^2} \dashint_{B_R} \lvert u-a_R\rvert^2 \,\d x ,
  \end{equation}
  with $\gamma(t) \colon [0,\infty) \to [0,1]$ a non-decreasing, continuous function such that $\gamma(0)=0,$ depending on $\omega_{\widetilde M}$ and $q$ only.
\end{lem}

This choice of $a_R$ is due to \textsc{Kronz} \cite{article:Kronz02}, whose significance is illustrated in the lemma below; this is essentially contained in \cite[Lemma 2(ii)]{article:Kronz02}, applying the Poincar\'e inequality in $W^{1,1}$ instead.

\begin{lem}\label{lem:interior_affinechoice}
  If $u \in W^{1,2}(B_R,\bb R^N)$, we have $a_R$ defined as in (\ref{eq:u_affine_approximant}) satisfies
  \begin{equation}\label{eq:interior_affineminima}
    \dashint_{B_R} \lvert u - a_R\rvert^2 \,\d x \leq \dashint_{B_R} \lvert u - a\rvert^2 \,\d x
  \end{equation} 
  for any $a : \bb R^n \to \bb R^N.$ Further we have the estimate
  \begin{equation}\label{eq:interior_affinederivative}
    \lvert \nabla a_R - (\nabla u)_{B_R}\rvert \leq C(n)\,\dashint_{B_R} \lvert\nabla u - (\nabla u)_{B_R(x_0)}\rvert \,\d x.
  \end{equation} 
  In particular if $\nabla u \in \BMO(B_R,\bb R^{Nn})$ we have $\lvert \nabla a_R - (\nabla u)_{B_R}\rvert \leq C(n) \seminorm{\nabla u}_{\BMO(B_R)}.$
\end{lem}

\begin{proof}[Proof of Lemma {\ref{lem:autonomous_caccioppoli}}]
  Set $\widetilde F(z) = F_{\nabla a_R}(z)$ as in (\ref{eq:autonomous_shifted}), and note by Lemma \ref{lem:interior_affinechoice} that
  \begin{equation}\label{eq:widetildeM}
    \lvert \nabla a_R\rvert \leq \lvert(\nabla u)_{B_R}\rvert + C(n)\seminorm{\nabla u}_{\BMO(B_R)} \leq \widetilde M.
  \end{equation} 
  Also fix a cutoff $\eta \in C^{\infty}_c(B_R)$ such that $1_{B_{R/2}} \leq \eta \leq 1_{B_R}$ and $\lvert \nabla \eta\rvert \leq \frac{C}{R}.$  Putting $w = u-a_R$ we have $w$ is $\widetilde F$-extremal since $u$ is $F$-extremal, and so testing the equation against $\phi = \eta^2w$ gives
  \begin{equation}\label{eq:autonomous_shifted_EL}
    0 = \dashint_{B_R} \widetilde F'(\nabla w) :\nabla(\eta^2w) \,\d x.
  \end{equation}
  Also since $\widetilde F''(0) = F''(\nabla a)$ satisfies the strict Legendre-Hadamard condition (\ref{eq:autonomous_quantitative_LH}) with $\lvert \nabla a\rvert\leq {\widetilde M},$ applying this to $\eta w \in W^{1,2}_0(\Omega,\bb R^N)$ gives (see for instance \cite[Theorem 10.1]{book:Giusti03}),
  \begin{equation}\label{eq:autonomous_interior_coercivity}
    \lambda_{\widetilde M} \dashint_{B_R} \lvert \nabla(\eta w)\rvert^2 \,\d x \leq \dashint_{B_R} \widetilde F''(0)\nabla(\eta w) : \nabla(\eta w) \,\d x.
  \end{equation}
  Taking the difference of (\ref{eq:autonomous_shifted_EL}), (\ref{eq:autonomous_interior_coercivity}) and rearranging we get
  \begin{equation}
    \begin{split}
      &\lambda_{\widetilde M} \dashint_{B_R} \lvert \nabla(\eta w)\rvert^2 \,\d x \\
      &\qquad\leq \dashint_{B_R} \eta^2 \left( \widetilde F''(0)\nabla w - \widetilde F'(\nabla w)\right): \nabla w \,\d x\\
      &\qquad\quad+ \dashint_{B_R} 2 \widetilde F''(0)(w\nabla \eta) :  (2\nabla (\eta w) - w\nabla \eta)\,\d x  - \dashint_{B_R} \widetilde F'(\nabla w):(2 \eta w \nabla \eta) \,\d x\\
      &\qquad\leq K_{\widetilde M} \dashint_{B_R} \omega_{\widetilde M}(\lvert \nabla w\rvert) \left(\lvert \nabla w\rvert^2 + \lvert \nabla w\rvert^q \right) \,\d x \\
      &\qquad\quad + 4K_{\widetilde M}\dashint_{B_R} \lvert w \nabla \eta\rvert \left( \lvert w \nabla \eta\rvert + \lvert \nabla(\eta w)\rvert + \eta \lvert \nabla w\rvert^{q-1} \right) \,\d x,
    \end{split}
  \end{equation}
  where we have used the comparison estimate (\ref{eq:autonomous_F_pertubation}) to control the first term along with the fact that $\eta^2 \leq 1,$ and the growth estimates (\ref{eq:autonomous_secondF_estimates}), (\ref{eq:autonomous_shifted_estimate2}) for the additional terms. We apply the modular Fefferman-Stein estimate (Corollary \ref{cor:BMO_modulus}) to the first term, noting that $\nabla w = \nabla u - (\nabla u)_{B_R}$ so
  \begin{equation}
    \begin{split}
      &\dashint_{B_R} \omega_{\widetilde M}(\lvert \nabla w\rvert) \left(\lvert \nabla w\rvert^2+\lvert \nabla w\rvert^q\right) \,\d x \\
      &\quad \leq C\, \omega_{\widetilde M}\left(\seminorm{\nabla u}_{\BMO(B_R)} + \lvert\nabla a_R - (\nabla u)_{B_R}\rvert\right) \dashint_{B_R} \lvert \nabla w\rvert^2 + \lvert \nabla w\rvert^q \,\d x \\
      &\quad \leq C\, \omega_{\widetilde M}\left(\seminorm{\nabla u}_{\BMO(B_R)} \right) \dashint_{B_R} \lvert \nabla w\rvert^2 + \lvert \nabla w\rvert^q \,\d x,
    \end{split}
  \end{equation}
  where we have used Lemma \ref{lem:interior_affinechoice} along with the fact that $\omega_{\widetilde M}(ts) \leq t\, \omega_{\widetilde M}(s)$ for $t \geq 1,$ $s \geq 0$ in the last line. Hence combining these with the earlier estimate and using Young's inequality to absorb the $\lvert \nabla(\eta w)\rvert^2$ term we arrive at
  \begin{equation}
    \begin{split}
      \dashint_{B_{R/2}} \lvert \nabla w\rvert^2 \,\d x &\leq  C \, \omega_{\widetilde M}\left(\seminorm{\nabla u}_{\BMO(B_R)}\right) \dashint_{B_R} \lvert \nabla w\rvert^2 + \lvert \nabla w\rvert^q \,\d x \\
      &\quad+\frac C{R^2} \dashint_{B_R} \lvert w\rvert^2 \,\d x + C \dashint_{B_R}\lvert \nabla w\rvert^{2(q-1)} \,\d x.
    \end{split}
  \end{equation}
  Note that if $q=2,$ we do not get the $\lvert \nabla w\rvert^{2(q-1)}$ term. Otherwise by the John-Nirenberg inequality (Proposition \ref{prop:global_JohnNirenberg}) and Lemma \ref{lem:interior_affinechoice} we can bound
  \begin{align}
    \dashint_{B_R} \lvert \nabla w\rvert^{q} \,\d x &\leq C \seminorm{\nabla u}_{\BMO(B_R)}^{q-2} \dashint_{B_R} \lvert \nabla w\rvert^2 \,\d x, \\
    \dashint_{B_R} \lvert \nabla w\rvert^{2(q-1)} \,\d x &\leq C\seminorm{\nabla u}_{\BMO(B_R)}^{2(q-2)} \dashint_{B_R} \lvert \nabla w\rvert^2 \,\d x.
  \end{align}
  Therefore if we let $\gamma(t) = \min\{1,\left( \omega_{\widetilde M}(t)(1+t^{q-2}) + t^{2(q-2)}\right)\}$ (omitting the $t^{q-2}$ terms if $q=2$) we deduce that
  \begin{equation}
    \dashint_{B_{R/2}} \lvert \nabla w\rvert^2 \,\d x \leq C\,\gamma\left(\seminorm{\nabla u}_{\BMO(B_R)}\right) \dashint_{B_R} \lvert \nabla w\rvert^2 \,\d x + \frac{C}{R^2} \dashint_{B_R} \lvert w\rvert^2 \,\d x,
  \end{equation}
  as required.
\end{proof}

\begin{rem}\label{rem:interior_feffermanstein}
  We have referred to Section \ref{sec:BMO} for the John-Nirenberg and modular Fefferman-Stein estimates, however in the interior case they can also be deduced from the corresponding statements in the full space using a cutoff argument. We will omit the details, but the argument is similar to that found in \cite{article:Moser01}; in this case the modular estimate can be proved more simply via a good-$\lambda$ estimate (see \cite[Lemma 6.2]{article:KristensenTaheri03}).
\end{rem}

\subsection{Harmonic approximation and interior regularity}

Our second ingredient is a comparison estimate for solutions to the linearised system. The following duality argument is an adaptation of the estimate proved in \cite{article:GmeinederKristensen19}. The linear theory we need will straightforwardly follow from the strict Legendre-Hadamard condition satisfied by $F''(\nabla a),$ and we will refer the reader to \cite[Chapter 10]{book:Giusti03} for details.

\begin{lem}[Interior harmonic approximation]\label{lem:autonomous_harmonic_approximation}
  Suppose $F$ satisfies Hypotheses \ref{hyp:autonomous_F}, let $M >0,$ and suppose $u$ is $F$-extremal in some ball $B_R(x_0) \subset \Omega$ such that $\nabla u \in \BMO(B_R,\bb R^{Nn})$ with $\seminorm{\nabla u}_{\BMO(B_R)} \leq 1$ and $\lvert(\nabla u)_{B_R}\rvert \leq M.$ Then letting $\nabla a_R$ as in (\ref{eq:u_affine_approximant}), we have the unique solution $h \in W^{1,2}(B_R,\bb R^N)$ to the problem
  \begin{equation}
    \pdeproblem{-\div F''(\nabla a_R)\nabla h}{0}{B_R,}{h}{u-a_R}{\partial B_R,}
  \end{equation}
  satisfies the $L^2$ estimate
  \begin{equation}\label{eq:autonomous_comparision_standardl2}
    \dashint_{B_R} \lvert \nabla h\rvert^2 \,\d x \leq C\dashint_{B_R} \lvert \nabla u - \nabla a_R\rvert^2 \,\d x
  \end{equation}
  with $\widetilde M = C(n)M$ and $C=C(n,N,K_{\widetilde M}/\lambda_{\widetilde M})>0,$ and further the comparison estimate
  \begin{equation}
    \frac1{R^2} \dashint_{B_R} \lvert u-a_R-h\rvert^2 \,\d x \leq C\,\gamma\left(\seminorm{\nabla u}_{\BMO(B_{R})}\right)\dashint_{B_R} \lvert \nabla u - \nabla a_R\rvert^2 \,\d x,
  \end{equation}
  with $C=C(n,N,q,K_{\widetilde M}/\lambda_{\widetilde M})$ and some $\gamma \colon [0,\infty) \to [0,1]$ increasing and continuous such that $\gamma(0)=0,$ depending on $\omega_{\widetilde M}$ and $q$ only.
\end{lem}

\begin{proof}
  By replacing $F$ with $\lambda_{\widetilde M}^{-1}F,$ we can replace $(K_{\widetilde M},\lambda_{\widetilde M})$ with $(K_{\widetilde M}/\lambda_{\widetilde M},1),$ where $\widetilde M \geq 1$ as in \eqref{eq:widetildeM}.
  Put $w = u-a_R$ and $\widetilde F = F_{\nabla a}.$ Then the existence of a unique $h \in W^{1,2}_w(B_R,\bb R^N)$ follows from $L^2$-coercivity of $\widetilde F''(0)$ (see \cite[Theorem 10.1]{book:Giusti03}) which gives (\ref{eq:autonomous_comparision_standardl2}). Then for any $\phi \in W^{1,2}_0(B_R,\bb R^N)$ we have
  \begin{equation}\label{eq:autonomous_F_harmonic_firstsep}
    \begin{split}
      \dashint_{B_R} \widetilde F''(0)(\nabla w - \nabla h) :  \nabla \phi  \,\d x &= \dashint_{B_R} \left(\widetilde F''(0)\nabla w - \widetilde F'(\nabla w)\right) : \nabla \phi \,\d x \\
      &\leq K_{\widetilde M}\dashint_{B_R} \omega_{\widetilde M}(\lvert \nabla w\rvert) \left(\lvert \nabla w\rvert + \lvert \nabla w\rvert^{q-1}\right) \lvert \nabla \phi\rvert \,\d x,
    \end{split}
  \end{equation}
  where we have used the fact that $w$ is $\widetilde F$-extremal the comparison estimate (\ref{eq:autonomous_F_pertubation}). Now choose $\phi$ to be the unique solution in $W^{1,2}_0 \cap W^{2,2}(B_R,\bb R^N)$ to the problem
  \begin{equation}
    - \div \widetilde F''(0)\nabla\phi = w-h
  \end{equation}
  in $B_R$ (see \cite[Theorem 10.3]{book:Giusti03}), so in particular by symmetry of $\widetilde F''(0)$ this satisfies
  \begin{equation}
    \dashint_{B_R} \widetilde F''(0)(\nabla w - \nabla h) : \nabla \phi \,\d x = \dashint_{B_R} \lvert w-h\rvert^2 \,\d x.
  \end{equation}
  Moreover $\phi$ satisfies a $W^{2,2}$ estimate which combined with the Poincar\'e-Sobolev inequality (noting $(\nabla \phi)_{B_R} = 0$) gives
  \begin{equation}\label{eq:autonomous_phi_comparision}
    \norm{\nabla \phi}_{L^{2^*}(B_R)} \leq C(n)\norm{\nabla^2 \phi}_{L^{2}(B_R)} \leq C\norm{w-h}_{L^{2}(B_R)}
  \end{equation}
  where $2^* = \frac{2n}{n-2}$ provided $n > 2.$ For this choice of $\phi,$ applying H\"older's inequality and rearranging (\ref{eq:autonomous_F_harmonic_firstsep}) using (\ref{eq:autonomous_phi_comparision}) we get
  \begin{equation}
    \frac1{R^2} \dashint_{B_R} \lvert w-h\rvert^2 \,\d x \leq \frac{C}{R^2} \norm{\omega_{\widetilde M}(\lvert \nabla w\rvert)}_{L^n(B_R)}^2 \dashint_{B_R} \lvert \nabla w\rvert^2 + \lvert \nabla w\rvert^{2(q-1)} \,\d x.
  \end{equation}
  If $n=2$ we use the fact that $\norm{\nabla\phi}_{L^4(B_R)} \leq C R^{\frac12} \norm{w-h}_{L^2(B_R)}$ to get the slightly modified estimate
  \begin{equation}
    \frac1{R^2} \dashint_{B_R} \lvert w-h\rvert^2 \,d x = \frac CR \norm{\omega_{\widetilde M}(\lvert \nabla w\rvert)}_{L^4(B_R)}^2 \dashint_{B_R} \lvert \nabla w\rvert^2 + \lvert \nabla w\rvert^{2(q-1)} \,\d x.
  \end{equation}
  In both cases since $\omega_{\widetilde M} \leq 1$ is concave by Jensen's inequality we have
  \begin{equation}
    \norm{\omega_{\widetilde M}(\lvert \nabla w\rvert)}_{L^p(B_R)} \leq R^{\frac np}\left( \dashint_{B_R} \omega_{\widetilde M}(\lvert \nabla w\rvert) \,d x\right)^{\frac1p}\leq R^{\frac np} \omega_{\widetilde M}\left(\seminorm{\nabla w}_{\BMO(B_R)}\right)^{\frac1p},
  \end{equation}
  and by the John-Nirenberg estimate (Proposition \ref{prop:global_JohnNirenberg}) and Lemma \ref{lem:interior_affinechoice} we can also estimate
  \begin{equation}
    \dashint_{B_R} \lvert \nabla w\rvert^{2(q-1)} \,\d x \leq C \seminorm{\nabla u}_{\BMO(B_{R})}^{2(q-2)} \dashint_{B_R} \lvert \nabla w\rvert^ 2\,\d x.
  \end{equation}
  Putting everything together the result follows by taking $\gamma(t) = \min\{1,\omega_{\widetilde M}(t)^{\frac2n} \left(1+t^{2(q-2)}\right)\},$ modified suitably if $n=2.$
\end{proof}

From here Theorem \ref{thm:autonomous_regularity} follows by combining the above estimate to get a suitable decay estimate, which can be applied iteratively. This approach is standard among many partial regularity proofs, and we follow a similar argument to that found in \cite{article:GmeinederKristensen19}.

\begin{proof}[Proof of Theorem {\ref{thm:autonomous_regularity}}]
  We will begin by establishing the following decay estimate for the excess energy
  \begin{equation}
    E(x,r) = \dashint_{B_r(x)} \lvert \nabla u - (\nabla u)_{B_R}\rvert^2 \,\d y.
  \end{equation}
  
  \textbf{Claim}: For any $B_{r}(x) \subset B_{R}(x_0)$ and $\sigma \in \left(0,\frac14\right)$ for which $\lvert (\nabla u)_{B_{2\sigma r}(x)}\rvert,\lvert (\nabla u)_{B_r(x)}\rvert\leq 2^{n+1}M$ and $\seminorm{\nabla u}_{\BMO_R(x)} \leq 1$ we have
  \begin{equation}\label{eq:autonomous_excessdecay}
    E(x,\sigma r) \leq C \left( \sigma^2 + \sigma^{-(n+2)}\gamma\left(\seminorm{\nabla u}_{\BMO(B_r(x))}\right)\right) E(x,r),
  \end{equation}
  where $C=C(n,N,q,K_M/\lambda_M)>0$ and $\gamma$ satisfies both Lemmas \ref{lem:autonomous_caccioppoli} and \ref{lem:autonomous_harmonic_approximation} with $C_*(n) M$ in place of $M.$

  Indeed let $a_r$ be as in (\ref{eq:u_affine_approximant}) centred at $x,$ and apply the harmonic approximation result (Lemma \ref{lem:autonomous_harmonic_approximation}) in $B_r(x)$ to get $h \in W^{1,2}_{u-a_r}(B_r(x),\bb R^N)$ solving
  \begin{equation}
    -\div F''(\nabla a_r)\nabla h = 0,
  \end{equation}
  which satisfies
  \begin{equation}\label{eq:autonomous_comparision_rest}
    \frac1{r^2}\dashint_{B_r(x)} \lvert u-a_r-h\rvert^2 \,\d y \leq \gamma\left(\seminorm{\nabla u}_{\BMO(B_R(x_0))}\right) E(x,r).
  \end{equation}
  Now letting $a_h(y) = h(x) + \nabla h(x) \cdot (y-x),$ since $B_{2\sigma r}(x) \subset B_{r/2}(x)$ we have
  \begin{equation}\label{eq:autonomus_comparision_harmonic}
    \begin{split}
      \frac1{(2\sigma r)^2}\dashint_{B_{2\sigma r}(x)} \lvert h - a_h\rvert^2 \,\d y &\leq \frac{C}{(2\sigma r)^2} \left(\sup_{B_{r/2}(x)} \lvert \nabla^2 h\rvert\right)^2\dashint_{B_{2\sigma r}(x)} \lvert y-x\rvert^4  \,\d y \\
      &\leq C\sigma^2 \dashint_{B_r(x)} \lvert \nabla h\rvert^2 \,\d y \leq C \sigma^2 E(x,r)
    \end{split}
  \end{equation}
  using interior regularity for $h$ (see for instance \cite[Theorem 10.7]{book:Giusti03}). We will use these in conjunction with the Caccioppoli-type inequality (Lemma \ref{lem:autonomous_caccioppoli}) applied in $B_{2\sigma r}(x),$ letting  $a_{2\sigma r}$ be given by (\ref{eq:u_affine_approximant}) we have
  \begin{equation}
    \begin{split}
      E(x,\sigma r) &\leq \dashint_{B_{\sigma r}(x)} \lvert \nabla u - \nabla a_{2\sigma r}\rvert^2 \,\d y \\
      &\leq \frac{C}{(2\sigma r)^2}\dashint_{B_{2\sigma r}(x)} \lvert u-a_{2\sigma r}\rvert^2 \,\d y + C \gamma\left(\seminorm{\nabla u}_{\BMO(B_r(x))}\right) E(x,2\sigma r).
    \end{split}
  \end{equation}
  Now using the estimates (\ref{eq:autonomous_comparision_rest}) and (\ref{eq:autonomus_comparision_harmonic}) and the minimising property (\ref{eq:interior_affineminima}) we can estimate
  \begin{equation}
    \begin{split}
      &\frac1{(2\sigma r)^2} \dashint_{B_{2\sigma r}(x)} \lvert u-a_{2\sigma r}\rvert^2 \,\d y \\
      &\quad\leq \frac1{(2\sigma r)^2} \dashint_{B_{2\sigma r}(x)} \lvert u-a_0 -a_1\rvert^2 \,\d y \\
      &\qquad\leq\frac C{\sigma^{n+2}r^2} \dashint_{B_{r}(x)} \lvert u-a_r-h\rvert^2 \,\d y + \frac C{(2\sigma r)^2} \dashint_{B_{2\sigma r}(x)} \lvert h - a_h\rvert^2 \,\d y \\
      &\quad\leq C\left(\sigma^2 + \sigma^{-(n+2)}\gamma\left( \seminorm{\nabla u}_{\BMO(B_R(x_0))}\right) \right) E(x,r).
    \end{split}
  \end{equation}
  So the claim follows by combining the above two estimates.

  We now iteratively apply the claim for suitably chosen parameters. Since $\lvert (\nabla u)\rvert_{B_R(x_0)} \leq M,$ for all $x \in B_{R/2}(x_0)$ we have $\lvert (\nabla u)_{B_{R/2}(x)}\rvert \leq M$ and so
  \begin{equation}
    \lvert (\nabla u)_{B_{\sigma R/2}(x)}\rvert \leq \lvert (\nabla u)_{B_{R/2}(x)}\rvert + \lvert (\nabla u)_{B_{\sigma R/2}(x)} - (\nabla u)_{B_{R/2}(x)}\rvert \leq 2^nM + \sigma^{-n} E(x,R/2).
  \end{equation}
  Iteratively applying this therefore gives
  \begin{equation}\label{eq:autonomous_averages_control}
    \lvert (\nabla u)_{B_{\sigma^k R/2}(x)}\rvert \leq 2^nM + \sigma^{-n}\sum_{j=0}^{k-1}  E(x,\sigma^jR/2).
  \end{equation}
  Since $E(x,r) \leq \seminorm{\nabla u}_{\BMO(B_R(x_0))}^2 \leq \eps^2$ for all $r < R/2,$ see that if $\sigma^{-n}\eps^2 \leq 2^nM,$ we can apply the claimed decay estimate (\ref{eq:autonomous_excessdecay}) to obtain 
  \begin{equation}
    E(x,\sigma R) \leq C \left( \sigma^2 + \sigma^{-n-2} \gamma(\eps)\right) E(x,R/2).
  \end{equation}
  Fix $\alpha \in (0,1),$ and choose $\sigma \leq \frac14$ such that $C\sigma^2 \leq \frac12 \sigma^{2\alpha}.$ Then we can take $\eps>0$ small enough so $C \sigma^{-(n+2)} \gamma(\eps) \leq \frac12 \sigma^{2\alpha}$ and $\sigma^{-n}\eps^2 \sum_j \sigma^{\alpha j} \leq 2^nM.$ Then we can inductively check that (\ref{eq:autonomous_excessdecay}) gives
  \begin{equation}
    E(x,\sigma^k R/2) \leq \sigma^{2\alpha k} E(x,R/2),
  \end{equation}
  and by (\ref{eq:autonomous_averages_control}) we can ensure
  \begin{equation}
    \lvert (\nabla u)_{B_{\sigma^k R/2}(x)}\rvert \leq 2^{n+1}M
  \end{equation}
  for each $k\geq 1.$ Hence for each $r \in (0,R/2),$ choosing $k$ such that $\sigma^k R/2 \leq r < \sigma^{k-1}R/2$ we deduce that
  \begin{equation}
    E(x,r) \leq C \left(\frac{r}{R}\right)^{2\alpha} E(x_0,R).
  \end{equation}
  This verifies the Campanato-Meyers characterisation of H\"older continuity (see for instance \cite[Theorem 2.9]{book:Giusti03}), allowing us to conclude that $u \in C^{1,\alpha}(\overline{B_{R/2}(x_0)},\bb R^N)$ as required.
\end{proof}

\section{Preliminaries for boundary regularity}\label{sec:technicalprelims}

Before we consider the boundary case, we will collect some technical results which will be used in our subsequent regularity proofs. While these results are largely known, some care was needed in keeping track of the associated constants.

\subsection{\texorpdfstring{$\BMO$}{BMO} in domains}\label{sec:BMO}

We will review some preliminary results about $\BMO$ functions and fix our conventions. For any $D \subset \bb R^n$ open, we define the \emph{Fefferman-Stein maximal function} associated to $f \in L^1_{\loc}(D,\bb R^{Nn})$ as
\begin{equation}
  \cal M_{D}^{\#}f(x) = \sup_{x \in B \subset D} \dashint_{B} \lvert f-(f)_{B}\rvert \,\d y,
\end{equation}
where we are taking the supremum over balls $B.$ Using this we can define the John-Nirenberg space $\BMO(D,\bb R^{Nn})$ of functions of \emph{bounded mean oscillation} in $D$ as the space of $f \in L^1_{\loc}(D,\bb R^{Nn})$ for which $\cal M_{D}^{\#}f \in L^{\infty}(D).$ We equip this space with the seminorm $\seminorm{f}_{\BMO(D)} = \lVert M_{D}^{\#}f\rVert_{L^{\infty}(D)}.$

While we wish to apply the results in this section to domains which are piecewise $C^{1,\beta},$ in order to understand the dependence of constants on the domain $D$ it will be convenient to work with \emph{John domains}; these were first introduced by \textsc{John} \cite{article:John61} and later named by \textsc{Martio \& Sarvas} \cite{article:MartioSarvas79}. The definition given here is slightly different to what appeared in the original papers, but can be found for instance in \cite{article:SmithStegenga90}.

\begin{defn}
  For $\delta \in (0,1)$ we say bounded domain $D \subset \bb R^n$ is a $\delta$-\emph{John domain} if there exists $x_0 \in D,$ called the \emph{John centre}, such that for all $x \in D$ there is a rectifiable curve $\gamma \colon [0,d] \to D$ parametrised by arclength such that $\gamma(0)=x,$ $\gamma(d) = x_0,$ and
  \begin{equation}
    \dist(\gamma(t),\partial D)  \geq \delta t
  \end{equation}
  for all $t \in [0,d].$
\end{defn}

This can be viewed as a \emph{twisted cone condition}, and since bounded Lipschitz domains satisfy a uniform cone condition (see for instance \cite[Section 4.4]{book:AdamsFournier03}) it follows that they are John domains. Moreover we have the following localisation property.

\begin{prop}\label{prop:local_john}
  Let $\Omega$ be a $C^1$ domain. Then there is $R_0 > 0$ and $\delta=\delta(n)>0$ such that for all $x_0 \in \overline\Omega$ and $0<R<R_0,$ we have $\Omega_R(x_0)$ is a $\delta$-John domain.
\end{prop}

We will postpone the proof to Section \ref{sec:boundary_averages}, which may be of independent interest.
This is the main reason why we have introduced these domains; if we can establish estimates in domains $\Omega_R(x_0)$ where the associated constant only depends on $\delta,$ then the constant holds uniformly among these domains.
This is particularly useful for our purposes where the dependence on $\delta$ naturally enters when considering estimates involving $\BMO.$
However we believe this is more generally a useful way to keep track of the constants for various technical estimates applied on $\Omega_R(x_0),$ which may not be easily controlled by the Lipschitz norm.

The first result we need is a global version of the John-Nirenberg inequality, which was proved in greater generality by \textsc{Smith \& Stegenga} \cite{article:SmithStegenga91} and \textsc{Hurri-Syrj\"anen} \cite{article:Hurri-Syrjanen93}. We will sketch the proof to clarify the dependence of constants.

\begin{prop}[Global John-Nirenberg estimate {\cite{article:SmithStegenga91,article:Hurri-Syrjanen93}}]\label{prop:global_JohnNirenberg}
  Suppose $D$ is a bounded $\delta$-John domain, and $f \in \BMO(D,\bb R^{Nn}).$ Then for all $1 \leq p < \infty,$ there is $C = C(n,p,\delta) > 0$ such that
  \begin{equation}
    \left(\dashint_{D} \lvert f-(f)_{D}\rvert \,\d x\right)^{\frac1p} \leq C\seminorm{f}_{\BMO(D)}.
  \end{equation}
\end{prop}

\begin{proof}[Proof sketch]
  The strategy is to take a Whitney decomposition $W = \{ Q_j\}$ of $D$ as given in \cite[Section VI.1]{book:stein_differentiability}, and apply the John-Nirenberg inequality on each $Q_j,$ which is is easily adapted from the original argument in \cite{article:JohnNirenberg61} (see also \cite[Corollary 2.2]{book:Giusti03}). To patch these local estimates we can use Whitney chains following \cite{article:Jones80} to show that 
  \begin{equation}
      \int_D \lvert f-(f)_D\rvert^p \,\d x \leq C(n,p)\left( \cal L^n(D) + \int_D k_D(x_0,x)^p \,\d x\right) \seminorm{f}_{\BMO(D)}^p,
  \end{equation}
  for a distinguished point $x_0 \in D,$ where $k_D$ is the \emph{quasi-hyperbolic distance} introduced in \cite{article:GehringPalka76} defined by
  \begin{equation}
    k_{D}(x_1,x_2) = \inf_{\gamma} \int_{\gamma} \frac1{\dist(x,\partial\Omega)} \,\d t,
  \end{equation}
  taking the infimum over all rectifiable curves $\gamma$ connecting $x_1, x_2 \in D.$ To verify the integrability of $k_D(x_0,\cdot)^p,$ letting $x_0$ be the John centre it is shown in \cite{article:GehringMartio85} that for all $x \in D,$
  \begin{equation}
    k_D(x_0,x) \leq \frac1{\delta} \log\frac{\dist(x_0,\partial\Omega)}{\dist(x,\partial\Omega)} + \frac1{\delta}\left(1+\log\left(1+\delta^{-1}\right)\right).
  \end{equation}
  Using this and keeping track of constants in the proof of \cite[Theorem 4]{article:SmithStegenga90} we have $k_D$ satisfies the integrability condition
  \begin{equation}
    \int_D k_D(x_0,x)^p \,\d x \leq C(n,p,\delta) \cal L^n(D),
  \end{equation}
  from which the result follows.
\end{proof}

We will also need an modular version of the Fefferman-Stein theorem \cite[Theorem 5]{article:FeffermanStein72} that holds up to the boundary. This estimate in the full space appeared in the work of \textsc{Kristensen \& Taheri} \cite{article:KristensenTaheri03} where is was proven by means of a good-$\lambda$ estimate, however to obtain estimates up to the boundary we will need a more refined approach using the extrapolation results of \textsc{Cruz-Uribe, Martell \& P\'erez} \cite{article:CruzUribeMartellPerez06}. We will briefly recall the notions of \emph{$N$-functions} considered in \cite{article:CruzUribeMartellPerez06}; these are mappings $\Phi \colon [0,\infty) \to [0,\infty)$ which are continuous, convex, and strictly increasing such that 
\begin{equation}
  \lim_{t \to 0^+} \frac{\Phi(t)}t = 0, \quad \lim_{t \to \infty} \frac{\Phi(t)}t = \infty.
\end{equation}
For such a $\Phi$ we can associate a conjugate function $\overline\Phi(t) = \sup_{s>0} \{st - \Phi(s)\},$ which can be shown to also be an $N$-function. We say  $\Phi \in \Delta_2$ if there is $C>0$ such that the doubling property $\Phi(2t) \leq C\Phi(t)$ holds, in which case the minimal $C$ will be denoted by $\Delta_2(\Phi).$ We also say $\Phi \in \nabla_2$ if $\overline\Phi \in \Delta_2$ and write $\nabla_2(\Phi) = \Delta_2(\overline\Phi);$ note this holds if there is $r>0$ such that $\Phi(rt) \geq 2r \Phi(t)$ for all $t \geq 0.$

\begin{prop}[Modular Fefferman-Stein estimate]\label{prop:fefferman_stein}
  Let $D \subset \bb R^n$ be a bounded $\delta$-John domain, and $\Phi$ an $N$-function such that $\Phi \in \Delta_2 \cap \nabla_2.$ Then there is $C=C(n,\delta,\Delta_2(\Phi),\nabla_2(\Phi))>0$ such that
  \begin{equation}
    \int_{D} \Phi(\lvert f-(f)_{D}\rvert) \,\d x \leq C \int_{D} \Phi\left(\cal M^{\#}_{D}f\right) \,\d x
  \end{equation}
  for all $f \in L^1_{\loc}(D,\bb R^{Nn})$ such that both sides are finite.
\end{prop}

This result is essentially proved in the work of \textsc{Diening, R\r{u}\v{z}i\v{c}ka, \& Schumacher} \cite{article:DRS10} in greater generality, however to obtain a modular estimate a slight modification is required in the proof.

\begin{proof}
  We first need a weighted $L^p$ estimate in $D$, so let $1<p<\infty$ and $w \in A_p.$ Then for any cube $Q$ it is shown in \cite[Corllary 7.2]{article:DRS10} that
  \begin{equation}
    \int_Q \lvert f\rvert^p w\,\d x \leq C\left(n,p,\seminorm{w}_{A_p}\right) \left( \int_Q \lvert \cal M_Q^{\#}f\rvert^p w \,\d x + \int_Q w \,\d x \left(\dashint_Q \lvert f\rvert \,\d x\right)^p\right)
  \end{equation}
  for all $f \in L^p(Q,w,\bb R^{Nn}),$ that is $f \colon Q \to \bb R^{Nn}$ such that $\lvert f\rvert^pw$ is integrable on $Q.$ 
  By applying this to $f - (f)_Q$ and noting that $\lvert f - (f)_Q\rvert \leq \cal M_Q^{\#}f$ we deduce that
  \begin{equation}
    \int_Q \lvert f-(f)_Q\rvert^p w\,\d x \leq C\left(n,p,\seminorm{w}_{A_p}\right) \int_Q \lvert \cal M_Q^{\#}f\rvert^p w \,\d x.
  \end{equation}
  To extend this to John domains we can apply \cite[Theorem 3]{article:IwaniecNolder85}; note it is proved in \cite[Lemma 2.1]{article:Boman82} that a $\delta$-John domain $D$ is a $\cal F(\sigma,N)$-domain as in \cite{article:IwaniecNolder85}, where $\sigma = \min\left\{\frac{10}9,\frac{n+1}n\right\}$ and $N=N(n,\delta).$
  Thus we obtain
  \begin{equation}
    \int_D \lvert f-(f)_{Q_0}\rvert^pw \,\d x \leq C\left(n,p,\seminorm{w}_{A_p},\delta\right) \int_D \lvert \cal M_D^{\#}f\rvert^p w \,\d x
  \end{equation}
  for all $1<p<\infty$ and $w \in A_p,$ for a distinguished cube $Q_0 \subset D.$ 
  A similar estimate appears in \cite[Theorem 5.23]{article:DRS10}, however the above is slightly sharper as we estimate $\lvert f - (f)_{Q_0}\rvert$ instead of $\lvert f - (f)_D\rvert$ which is important in the sequel.

  Now we apply the modular extrapolation theorem in \cite{article:CruzUribeMartellPerez06} (see also \cite[Chapter 4]{book:CruzUribeMartellPerez11}) to the family of pairs $(\lvert f-f_{Q_0}\rvert 1_{D}, \lvert M^{\#}_Df\rvert 1_D)$ we obtain
  \begin{equation}
    \int_{D} \Phi(\lvert f-(f)_{Q_0}\rvert) \,\d x \leq C(n,\delta,\Delta_2(\Phi),\nabla_2(\Phi)) \int_{D} \Phi\left(\cal M^{\#}_{D}f\right) \,\d x.
  \end{equation}
  Replacing the average $(f)_{Q_0}$ by $(f)_D$ using the doubling property and convexity of $\Phi,$ the results follows.
\end{proof}

We wish to apply this result to $\Phi(t) = \omega(t)t^p$ with $p>1,$ where $\omega \colon [0,\infty) \to [0,1]$ is a continuous, non-decreasing, concave function such that $\omega(0)=0$ as in Section \ref{sec:autonomous_integrand}. A technical complication arises as this need not be convex in general, but adapting a construction in \textsc{Kokilashvili \& Krbec} \cite{book:KokilashviliKrbec91} we can work with a modified $\widetilde\Phi$ which is convex instead. 

\begin{cor}\label{cor:BMO_modulus}
  Suppose $D \subset \bb R^n$ is a bounded $\delta$-John domain, $1<p<\infty,$ and $\omega \colon [0,\infty) \to [0,1]$ is non-decreasing, continuous, concave with $\omega(0)=0.$ Then if $f \in \BMO(D,\bb R^{Nn}),$ for each $1<p<\infty$ there is $C = C(n,p,\delta)>0$ such that
  \begin{equation}
    \int_{D} \omega(\lvert f-(f)_D\rvert)\lvert f-(f)_D\rvert^p \,\d x \leq C\,\omega\left(\seminorm{f}_{\BMO(D)}\right)\int_{D} \lvert f-(f)_D\rvert^p \,\d x.
  \end{equation}
\end{cor}

\begin{proof}
  We will first construct an $N$-function $\widetilde\Phi$ such that 
  \begin{equation}\label{eq:modified_phi_comparision}
    \widetilde\Phi(t) \leq \Phi(t) \leq \widetilde \Phi(2at)
  \end{equation}
  for all $t \geq 0,$ where $a \geq 1$ to be determined. 
  Since $\omega$ is increasing we have $\Phi(t) \leq \frac1{2a} \Phi(at)$ with $a = 2^{\frac1{p-1}},$ and so by \cite[Lemmas 1.1.1, 1.2.3]{book:KokilashviliKrbec91} we get
  \begin{equation}
    \widetilde\Phi(t) = \frac1a \int_0^{\frac ta} \sup_{0 < \tau < s} \left(\omega(\tau) \tau^{p-1}\right) \,\d s
  \end{equation}
  is convex and increasing on $[0,\infty)$ satisfying \eqref{eq:modified_phi_comparision}.
  Further since $\Phi$ satisfies $\Phi(2t) \leq 2^{p+1}\Phi(t)$ and $\Phi(at) \geq 2a\Phi(t)$ we can infer that $\widetilde\Phi(2t) \leq 2^{p+1}\widetilde\Phi(t)$ and $\widetilde\Phi(at) \geq 2a\Phi(t)$ also, so $\widetilde\Phi \in \Delta_2 \cap \nabla_2$ and the associated constants can be chosen to depend on $p$ only.

  Now applying Proposition \ref{prop:fefferman_stein} to $\widetilde\Phi$ and using \eqref{eq:modified_phi_comparision}, for $f \in L^p(D,\bb R^{Nn})$ we deduce that
  \begin{equation}\label{eq:feffermanstein_cor_int}
    \begin{split}
      \int_D \Phi\left( \lvert f - (f)_D\rvert \right) \,\d x
      &\leq C \int_D \Phi\left( \cal M_D^{\#}f \right) \,\d x\\
      &\leq C\,\omega(\norm{f}_{\BMO(D)}) \int_{\bb R^n} \lvert \cal M(f 1_D)\rvert^p \,\d x,
    \end{split}
  \end{equation} 
  where we have used the fact that $\cal M_D^{\#}f \leq \norm{f}_{\BMO}(D)$ and $\cal M_D^{\#}f \leq \cal M(f 1_D),$ where $\cal M$ is the Hardy-Littlewood maximal operator on $\bb R^n$ defined for $g \in L^1_{\loc}(\bb R^n)$ by
  \begin{equation}
    \cal M(g)(x) = \sup_{B \ni x} \dashint_{B} \lvert f\rvert \,\d y,
  \end{equation} 
  taking the supremum over all balls $B \subset \bb R^n$ containing $x.$
  The Hardy-Littlewood maximal theorem asserts $\cal M$ is bounded on $L^p(\bb R^n)$ for $1<p \leq \infty$ (see for instance \cite[Theorem I.1]{book:stein_differentiability}), so applying this with $g = f\, 1_D$, \eqref{eq:feffermanstein_cor_int} becomes
  \begin{equation}
    \int_D \Phi\left( \lvert f - (f)_D\rvert \right) \,\d x \leq C \omega(\norm{f}_{\BMO(D)}) \int_D \lvert f\rvert^p \d x
  \end{equation} 
  as required.
\end{proof}

\subsection{Localisation near the boundary}\label{sec:boundary_averages}

For the Caccioppoli-type estimate in the interior (Lemma \ref{lem:autonomous_caccioppoli}), our strategy involved testing the equation against $\phi = \eta(u-a)$ with $\eta$ a cutoff and $a$ an affine approximation to $u$ in a ball. This will need to be modified for the boundary case to ensure our test function $\phi$ vanishes on $\partial\Omega.$ In this section we collect the necessary technical ingredients to construct a suitable replacement function, using ideas of \textsc{Kronz} \cite{article:Kronz05} along with the refinements of \textsc{Campos Cordero} \cite{article:CamposCordero17}.

Let $\Omega \subset \bb R^n$ be a bounded $C^{1,\beta}$ domain, that is, $\partial\Omega$ can locally be written as the graph of a $C^{1,\beta}$ function in the following sense; for all $x_0 \in \partial\Omega,$ there is $R_0>0$ and a unit vector $\nu_{x_0} \in \bb R^n$ such that letting $T_{x_0}= \langle \nu_{x_0} \rangle^{\perp}$ denote the orthogonal complement, there is a map
\begin{equation}
  \gamma \colon T_{x_0} \cap B_{R_0} \rightarrow \bb R
\end{equation}
which is of class $C^{1,\beta}$ such that we have $\nabla\gamma(0) = 0$ and
\begin{align}
  \Omega \cap B_{R_0}(x_0) &= B_{R_0}(x_0) \cap \left\{ x_0 + y + \lambda \nu : y \in T_{x_0} \cap B_{R_0}, \lambda < \gamma(y)\right\}, \\
  \partial\Omega \cap B_{R_0}(x_0) &= B_{R_0}(x_0) \cap \left\{x_0+ y + \gamma(y) \nu : y \in T_{x_0} \cap B_{R_0}\right\}.
\end{align}
Note this also allows us to define Lipschitz domains and $C^{k,\beta}$ domains analogously.
In the $C^{1,\beta}$ case, this implies there is an outward facing unit normal $\nu_{\partial\Omega}$ given by $\nu_{\partial\Omega}(x_0) = \nu_{x_0}$ at each $x_0 \in \partial\Omega.$ This also allows us to construct a \emph{defining function} $\rho = \rho_{\Omega} \in C^{1,\beta}(\bb R^n)$ with the property that
\begin{equation}
  \Omega = \{ x \in \bb R^n : \rho(x) < 0\},\quad \bb R^n \setminus \overline\Omega = \{x \in \bb R^n : \rho(x)>0\},
\end{equation}
and such that $\nabla \rho(x) \neq 0$ in $\partial\Omega,$ by locally defining $\rho(x) = \langle (x-x_0), \nu \rangle - \gamma(x-x_0) $ in $B_{R_0}(x_0)$ and patching using a partition of unity. Note that $\nabla \rho(x)$ is normal to $\partial\Omega$ at each $x \in \partial\Omega,$ so we have $\nu_{\partial\Omega}(x) = \frac{\nabla\rho(x)}{\lvert \nabla \rho(x)\rvert}.$ We also define the associated $C^{1,\beta}$-constant of $\Omega$ as
\begin{equation}
  \norm{\Omega}_{C^{1,\beta}} = \inf\left\{\sup_{1 \leq j \leq N} \norm{\nabla\gamma_j}_{C^{0,\beta}(T_{x_j} \cap B_{R_j}(x_j))}\right\},
\end{equation}
where the infimum is taken over collections $\{\gamma_j,x_j,R_j\}_{j=1}^N$ where $\{B_{R_j}(x_j)\}$ covers $\partial\Omega$ and each $\Omega \cap B_{R_j}(x_j)$ is represented as the graph of the $C^{1,\beta}$ function $\gamma_j.$

The idea is to use this defining function $\rho$ as a replacement for the affine approximation, considering maps of the form
\begin{equation}
  a(x) = \xi\, \frac{\rho(x)}{\lvert \nabla \rho(x_0)\rvert},
\end{equation}
with $\xi \in \bb R^N.$ Since $\nabla a = \xi \tensor \frac{\nabla\rho(x)}{\lvert \nabla\rho(x_0\rvert}$ which is close to $\xi \tensor \nu_{x_0}$ however, taking $\xi = (\nabla v \cdot \nu_{x_0})_{\Omega_R(x_0)}$ only allows us to control the normal component compared to the full derivative $\nabla a = (\nabla u)_{B_R(x_0)}$ from the interior case. It turns out this is sufficient however; this is illustrated by the following result, which is an adaptation of an observation of \textsc{Campos Cordero} \cite{article:CamposCordero17}.

\begin{lem}\label{lem:boundary_aconstruction}
  Let $\Omega \subset \bb R^n$ be a bounded $C^{1,\beta}$ domain and let $p > \frac32.$ There is $R_0 > 0$ and $C>0$ such that for all $x_0 \in \partial\Omega$ and $0 < R < R_0,$ for all $v \in W^{1,p}(\Omega_R(x_0),\bb R^N)$ such that $v = 0$ on $\partial\Omega \cap B_R(x_0)$ we have
  \begin{equation}
    \begin{split}
      &\left(\dashint_{\Omega_{R}(x_0)} \lvert \nabla v - (\nabla v \cdot \nu_{x_0})_{\Omega_R(x_0)} \tensor \nu_{x_0}\rvert^p \,\d x\right)^{\frac1p} \\
      &\qquad\leq C \left( \dashint_{\Omega_R(x_0)} \lvert \nabla v - (\nabla v)_{\Omega_R(x_0)}\rvert^p \,\d x\right)^{\frac1p} + C \lvert (\nabla v)_{\Omega_R(x_0)}\rvert R^{\beta}.
    \end{split}
  \end{equation}
\end{lem}

\begin{proof}
  Fix $x_0 \in \partial\Omega,$ then by translating and rotating we can assume $x_0=0$ and $\nu(x_0)=e_n,$ and take $R_0>0$ small enough so we can write $\Omega_{R_0}(x_0)$ as the graph of some $\gamma.$ We have
  \begin{equation}\label{eq:camposcordero_firstestimate}
      \left( \dashint_{\Omega_R} \lvert \nabla v - (\nabla_n v)_{\Omega_R} \tensor e_n \rvert^p \,\d x \right)^{\frac1p}  \leq \left( \dashint_{\Omega_R} \lvert \nabla v - (\nabla v)_{\Omega_R}\rvert^p \,\d x \right)^{\frac1p} + \sum_{i=1}^{n-1} \lvert (\nabla_iv)_{\Omega_R}\rvert,
  \end{equation}
   where we write $\nabla_j v = \nabla v \cdot e_j,$ so we need to estimate the tangential derivatives. We proceed analogously to \cite[Lemma 5.6]{article:CamposCordero17} with minor modifications to account for the curved boundary, so letting $\rho$ be the defining function for $\Omega$ as above we consider
   \begin{equation}
     \widetilde v(x) = v(x) - (\nabla_nv)_{\Omega_R} \frac{\rho(x)}{\lvert \nabla \rho(0)\rvert}.
   \end{equation}
   Note that $\widetilde v$ still vanishes on $\partial\Omega \cap B_R,$ so writing $x=(x',x_n) \in \bb R^{n-1} \times \bb R,$ so a similar argument to \cite{article:CamposCordero17} gives
   \begin{equation}\label{eq:cordero_divergence}
      \int_{\Omega_R} \nabla_i \widetilde v \,\d x = \int_{\Omega \cap \partial B_R} \widetilde v(x) \frac{x_i}R \,\d \cal H^{n-1}(x) = \int_{\Omega_R} \nabla_n\widetilde v(x) \frac{x_i}{(R^2-\lvert x'\rvert^2)^{\frac12}} \,\d x,
  \end{equation}
  where the only difference is that $\widetilde v$ vanishes at $(x',\gamma(x'))$ writing $x=(x',x_n).$ This can then be estimated using H\"older's inequality as in \cite{article:CamposCordero17} to get
  \begin{equation}\label{eq:cordero_onederivative}
    \left\lvert\dashint_{\Omega_R} \nabla_i \widetilde v \,\d x\right\rvert \leq \left(\dashint_{\Omega_R} \lvert\nabla_n \widetilde v\rvert^p \,\d x\right)^{\frac1p}.
  \end{equation} 
  Now using the fact that $\rho$ is of class $C^{1,\beta}$ we deduce that
  \begin{equation}
    \begin{split}
      \left\lvert \dashint_{\Omega_R} \nabla_iv \,\d x \right\rvert &\leq  \left\lvert \dashint_{\Omega_R} \nabla_i\widetilde v \,\d x \right\rvert + \lvert (\nabla_nv)_{\Omega_R}\rvert\frac{\lvert (\nabla_i\rho)_{\Omega_R}\rvert}{\lvert \nabla\rho(x_0)\rvert} \\
      &\leq C(n,p) \left(\dashint_{\Omega_R} \lvert \nabla_nv - (\nabla_nv)_{\Omega_R}\rvert^p \,\d x\right)^{\frac1p}+ C(n,p\norm{\Omega}_{C^{1,\beta}})\lvert (\nabla_nv)_{\Omega_R}\rvert R^{\beta},
    \end{split}
  \end{equation}
  where we used \eqref{eq:cordero_onederivative} in the second line. 
  Thus combining with \eqref{eq:camposcordero_firstestimate} the result follows.
\end{proof}

We close this subsection with the proof of Proposition \ref{prop:local_john}, which will be an consequence of the following more general result.

\begin{lem}\label{lem:lipschitz_localjohn}
  Let $\Omega$ be a Lipschitz-domain with $\norm{\Omega}_{C^{0,1}} <1.$ Then for all $x_0 \in \overline\Omega$ and $0<R<R_0,$ we have $\Omega_R(x_0) = \Omega \cap B_R(x_0)$ is a $\delta$-John domain, where $\delta$ can be chosen to depend on $n$ and $\norm{\Omega}_{C^{0,1}}$ only.
\end{lem}

\begin{proof}
  Put $L := \norm{\Omega}_{C^{0,1}} < 1.$
  Let $R_0>0$ such that $\Omega_R(x_0)$ can be written as the graph of a Lipschitz function $\gamma$ when $R<R_0.$
  By means of a rigid motion assume that $x_0 =0,$ $\nu(x_0) = -e_n$ and $T_{x_0} = H = \{x \in \bb R^n : x_n = 0\}.$
  Moreover by rescaling we can assume that $R=1,$ so we have
  \begin{equation}
    \Omega \cap B = \left\{ x \in B :  x_n > \gamma(x')\right\}. 
  \end{equation} 
  By assumption we have $\lvert \nabla\gamma\rvert \leq L$ a.e.\,in $H \cap B$ and $\gamma(0)=0$, which implies that $\lvert \gamma(x')\rvert \leq L \lvert x'\rvert.$ 
  Therefore noting $x_n = \gamma(x')$ on $\partial\Omega$ we have
  \begin{equation}\label{eq:coneunion_L}
    \partial\Omega \cap B \subset \left\{ x \in B : \lvert x'\rvert \geq \frac{\lvert x\rvert}{\sqrt{L^2+1}} \right\} =: S_{L}.
  \end{equation} 
  Moreover $S_L$ can be seen as the union of all cones 
  \begin{equation}
    C(n,\theta_L) := \left\{ x \in \bb R^n : \lvert x \cdot n\rvert \geq \lvert x \rvert\cos \theta_L \right\}
  \end{equation} 
  intersected with $B$ for all $n \in S^{n-2} \times \{0\},$ where $\cos \theta_L = 1/\sqrt{L^2+1}.$
  Note that $\theta_L < \frac{\pi}4$ if and only if $L < 1.$
  We will also let $S_1$ to be as in \eqref{eq:coneunion_L} where $L$ is replaced by $1$.
  Also since $\Omega \cap B \supset B^+ \setminus S_L$ where $B^+ = \{ x \in B : x_n > 0\},$ we will choose $y_0 = \frac12 e_n$ in $B^+ \setminus S_L$ to be our John centre.
  Since $B^+ \setminus S_L$ is convex, it is shown by \textsc{Martio \& Sarvas} \cite[Remark 2.4(c)]{article:MartioSarvas79} that it is a John domain with constant $\frac1{2\sqrt{2}}$, since $B_{\frac1{2\sqrt{2}}}(y_0) \subset B^+ \setminus S_1 \subset B^+ \setminus S_L \subset B.$

  Now let $x=(x',x_n) \in \Omega \cap S$, noting that $x' \neq 0$ necessarily.
  We wish to construct a piecewise linear path from $x$ to $y_0$ verifying the John domain assumption, as drawn in Figure \ref{fig:john_lipschitz}, which will involve some elementary geometry.

\begin{figure}[ht]
    \centering
    \begin{tikzpicture}[scale=.75]
	\begin{pgfonlayer}{nodelayer}
		\node [circle,fill,minimum size = 3pt,inner sep = 0pt, outer sep = 0pt] (0) at (0, 0) {};
		\node [style=none] (1) at (0, 9) {};
		\node [style=none] (2) at (10, 0) {};
		\node [style=none] (3) at (0, 8) {};
		\node [style=none] (4) at (9, 0) {};
		\node [style=none] (6) at (7, 7) {};
		\node [style=none] (9) at (9, 4.5) {};
		\node [style=none] (10) at (4, 0.75) {};
		\node [style=none] (11) at (5.75, 0.25) {};
		\node [style=none] (12) at (7.25, 0.75) {};
		\node [style=none] (13) at (9.5, 0.5) {};
		\node [style=none] (14) at (1, 1) {};
		\node [style=none] (15) at (2, 1) {};
		\node [style=none] (16) at (1.5, 0) {};
		\node [style=none] (20) at (2.25, 0) {};
		\node [style=none] (21) at (0.5, 9) {$x_n$};
		\node [style=none] (22) at (7.5, 7) {};
		\node [style=none] (23) at (7.5, 7) {$\partial S_1$};
		\node [style=none] (24) at (9.5, 4.5) {$\partial S_L$};
		\node [style=none] (25) at (9.9, 0.5) {$\partial\Omega$};
		\node [circle,fill=white,minimum size = 3pt,inner sep = 0pt, outer sep = 0pt] (28) at (1.7, 0.75) {$\frac{\pi}4$};
		\node [style=none] (29) at (2.5, 0.75) {$\theta_L$};
		\node [circle,fill,minimum size = 3pt,inner sep = 0pt, outer sep = 0pt] (30) at (8, 1) {};
		\node [style=none] (31) at (8.3, 1) {$x$};
		\node [style=none] (32) at (4.5, 4.5) {};
		\node [circle,fill,minimum size = 3pt,inner sep = 0pt, outer sep = 0pt] (34) at (0, 5) {};
		\node [style=none] (35) at (0.35, 5.25) {$y_0$};
		\node [style=none] (36) at (4, 7.7) {$B$};
		\node [style=none] (37) at (10.5, 0) {};
		\node [style=none] (38) at (10.4, 0) {$x'$};
		\node [style=none] (40) at (6.25, 2.75) {};
		\node [style=none] (41) at (2.25, 4.75) {};	\end{pgfonlayer}
		\node [style=none] (44) at (3.75, 4.9) {$\gamma$};
	\begin{pgfonlayer}{edgelayer}
		\draw [thick,->](0.center) to (1.center);
		\draw [thick,->](0.center) to (2.center);
		\draw [thick,bend left=45] (3.center) to (4.center);
    \draw [dashed](0.center) to (6.center);
		\draw [dashed](0.center) to (9.center);
    \draw [thick](0.center) to (10.center);
		\draw [thick,bend left, looseness=0.75] (10.center) to (11.center);
		\draw [thick,bend right] (11.center) to (12.center);
		\draw [thick,bend left=15] (12.center) to (13.center);
		\draw [bend left, looseness=0.75] (14.center) to (16.center);
		\draw [bend left=15] (15.center) to (20.center);
		\draw (30.center) to (32.center);
    \draw (32.center) to (34.center);
		\draw [-{Stealth[scale=2]}](32.center) to (41.center);
		\draw [-{Stealth[scale=2]}](30.center) to (40.center);
	\end{pgfonlayer}
\end{tikzpicture}
    \caption{Construction of the path $\gamma$}
   \label{fig:john_lipschitz}
\end{figure}
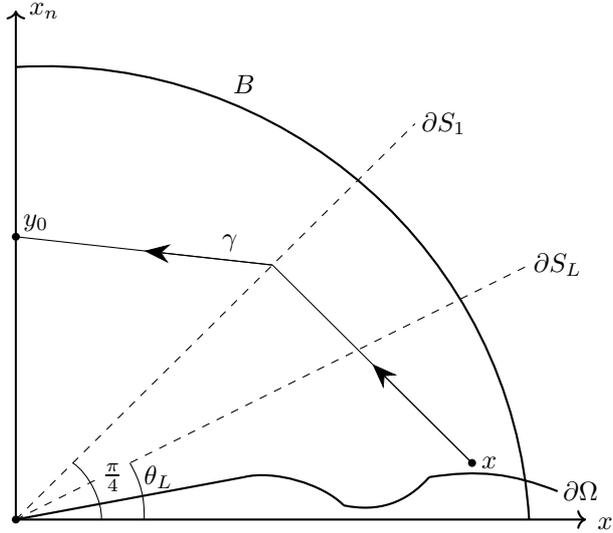

  Let $\omega = \frac{(x',0)}{\lvert x'\rvert}$ and put $x_t = x + \frac{t}{\lvert x\rvert} (-x_n \omega + \lvert x'\rvert e_n),$ which is parametrised by arclength.
  We also let $\theta_t \in (0,2\pi)$ such that $x_t = \omega \cos \theta_t + e_n \sin \theta_t$; note that $\theta_0 \in (\theta_L, \frac{\pi}4)$.
  We now claim that $\lvert x_t\rvert = \dist(x, \partial B_1)$ is linearly decreasing in $t$ provided $x_t \in S_1.$
  To see this, consider the triangle formed by the points $P = 0,$ $Q = x_0$ and $R = x_t;$ then we have the angles $\angle RPQ = \pi - (\frac{\pi}4 + \theta_t)$ and $\angle PQR = \theta_L + \frac{\pi}4$.
  Then $\lvert Q - P\rvert = \lvert x_0\rvert \leq 1$, $\lvert R - Q \rvert = \lvert x_t\rvert$, $\lvert R-Q\rvert = t$, and $\angle PQR = \frac{\pi}4 - \theta_0$.
  By the cosine rule we have
  \begin{equation}
    \lvert x_t\rvert^2= \lvert x\rvert^2 + t^2 - 2 t \cos\left( \frac{\pi}4 - \theta_0 \right) := p(t).
  \end{equation} 
  Let $t_0>0$ be the unique value such that $\theta_{t_0} = \frac{\pi}4$, which is where $x_t$ exists $S_1.$
  In this case, since $\angle QRP = \frac{\pi}2$ we have $t_0 = \lvert x\rvert \sin(\frac{\pi}4-\theta_0).$ 
  Note also that $\lvert x_{t_0}\rvert = \lvert x\rvert \cos(\frac{\pi}4-\theta_0)$.
  Therefore for $t \in (0,t_0)$ we have
  \begin{equation}
    p'(t) = 2 (t - \lvert x\rvert \cos\left( \frac{\pi}4-\theta_0 \right) \geq 2\lvert x\rvert \left(\cos\left( \frac{\pi}4-\theta_0 \right)-t_0\right) = - 2 \sqrt{2} \lvert x\rvert \sin(\theta_0).
  \end{equation} 
  Hence we deduce that
  \begin{equation}
    \lvert x \rvert - \lvert x_t\rvert = -\frac12 \int_0^t \frac{p'(s)}{\sqrt{p(s)}}\,\d s \geq \sqrt{2} \sin(\theta_0) \frac{\lvert x\rvert}{\lvert x_{t_0}\rvert} \geq \delta t := \frac{\sqrt{2} \sin(\theta_L)}{\cos\left(\frac{\pi}4-\theta_L\right)} t,
  \end{equation} 
  so it follows that
  \begin{equation}
    \dist(x_t,\partial B) = 1 - \lvert x_t\rvert \geq \delta t.
  \end{equation} 
  Also since $\gamma$ is $L$-Lipschitz, we have
  \begin{equation}
    \left(x + C\left(e_n,\frac{\pi}2 - \theta_L\right)\right) \cap B \subset \Omega \cap B.
  \end{equation} 
  Indeed if $y \in C(e_n,\frac{\pi}2-\theta_L)$ then $y_n > L\lvert y'\rvert$ and hence
  \begin{equation}
    x_n + y_n > \gamma(x') - L \lvert y'\rvert \geq \gamma(x'+y'),
  \end{equation} 
  so $x + y \in \Omega$ provided $\lvert x+y\rvert \leq 1$.
  Since $x_t$ lies in the cone $x + C(e_n,\frac{\pi}2)$, some more trigonometry gives
  \begin{equation}
    \dist(x_t,\partial\Omega) \geq \dist\left(x_t, x + \partial C\left(e_n,\frac{\pi}2-\theta_L\right)\right) = \lvert x_t - x\rvert \sin\left(\frac{\pi}4-\theta_L\right) = \sigma t,
  \end{equation} 
  where $\sigma = \sin\left( \frac{\pi}4-\theta_L \right).$
  Combining the above two estimates we deduce that
  \begin{equation}
    \dist(x_t, \partial (\Omega \cap B)) \geq \min\{\delta,\sigma\} t
  \end{equation} 
  for all $0<t<t_0.$
  We can then join $x_{t_0}$ to the John centre $y_0$ via a linear combination to conclude, which is also how the case $x \in \Omega \setminus S_L$ is treated.
\end{proof}

\subsection{Reference estimates up to the boundary}

We will also need some reference estimates for linear elliptic systems for the harmonic approximation step. We consider a linear mapping $\bb A: \bb R^{Nn} \to \bb R^{Nn}$ which is symmetric in the sense that $v : \bb A w = \bb Av : w,$ satisfying the uniform Legendre-Hadamard ellipticity condition
\begin{equation}
  \lambda \lvert \xi\rvert^2\lvert \eta\rvert^2 \leq \bb A(\xi \tensor \eta) : (\xi \tensor \eta) \leq \Lambda\lvert \xi\rvert^2\lvert \eta\rvert^2
\end{equation}
holds for all $\xi \in \bb R^N,$ $\eta \in \bb R^n$ with $\lambda>0.$ By means of the Fourier transform one can infer that for any $\Omega \subset \bb R^n$ open the estimate
\begin{equation}\label{eq:LH_coercivity}
  \int_{\Omega} \lvert \nabla \varphi\rvert^2 \,\d x \leq \frac1{\lambda}\int_{\Omega} \bb A \nabla \varphi : \nabla \varphi \,\d x
\end{equation}
holds for all $\varphi \in W^{1,2}_0(\Omega,\bb R^N),$ so the Lax-Milgram lemma gives the associated operator $-\div(\bb A \nabla\cdot) : W^{1,2}_0(\Omega,\bb R^N) \to W^{-1,2}(\Omega,\bb R^N)$ is an isomorphism.

In the interior case we considered the same setting, but we used uniform and $W^{2,2}$ estimates which could be found in many sources such a \cite{book:Giusti03}. For boundary regularity we wish to establish analogous estimates for $\Omega_R(x_0),$ however such domains are merely piecewise $C^{1,\beta}$ which is too weak to expect estimates in those scales. To circumvent this we will need to replace $\Omega_R$ by a suitably regular domain following an argument used by \textsc{Kristensen \& Mingione} in \cite{article:Kristensen_Mingione10}, and obtain weakened estimates which will be sufficient for our purposes.

\begin{lem}\label{lem:boundary_linearelliptic}
  Let $\Omega \subset \bb R^n$ be a bounded $C^{1,\beta}$ domain and let $\bb A$ be symmetric and uniformly Legendre-Hadamard elliptic as above. Then there is $R_0>0$ such that for each $x_0 \in \partial\Omega$ and $0 < R < R_0,$ there exists a $C^{1,\beta}$ domain $\widetilde\Omega_R(x_0)=\widetilde\Omega_R$ such that
  \begin{equation}
    \overline{\Omega_{R/2}(x_0)} \subset \widetilde\Omega_R \subset \Omega_{R}(x_0),
  \end{equation}
  on which the following solvability results hold.

  \begin{enumerate}[label=(\roman*)]

    \item If $v \in W^{1,2}(\widetilde\Omega_R(x_0))$ such that $v=0$ on $\partial\Omega \cap \partial\widetilde\Omega_R(x_0),$ the unique $h \in W^{1,2}_v(\widetilde\Omega_R(x_0))$ solving
      \begin{equation}
        \pdeproblem{-\div(\bb A\nabla h)}{0}{\widetilde\Omega_R(x_0),}{h}{v}{\partial\widetilde\Omega_R(x_0),}
      \end{equation}
      is of class $C^{1,\beta}$ in $\widetilde\Omega_R \cup \left(\partial\Omega \cap \partial\widetilde\Omega_R(x_0)\right)$ with the associated estimate
      \begin{equation}\label{eq:boundary_schauder}
        \seminorm{\nabla h}_{C^{1,\beta}(\overline{\Omega_{R/2}(x_0))}} \leq C(n,N,\Lambda/\lambda,\beta,\norm{\Omega}_{C^{1,\beta}}) \left(\dashint_{\widetilde\Omega_R} \lvert \nabla h\rvert^2 \,\d x\right)^{\frac12}.
      \end{equation}
      
    \item If $2 \leq p<\infty$ and $F \in L^p(\Omega,\bb R^{nN}),$ then there is a unique $u \in W^{1,p}_0(\Omega,\bb R^N)$ solving
      \begin{equation}
        \pdeproblem{-\div(\bb A\nabla u)}{-\div F}{\widetilde\Omega_R(x_0),}{u}{0}{\partial\widetilde\Omega_R(x_0),}
      \end{equation}
      which satisfies the estimate
      \begin{equation}\label{eq:boundary_lp}
        \int_{\widetilde\Omega_R(x_0)} \lvert \nabla u\rvert^p \,\d x \leq C(n,N,p,\Lambda/\lambda,\norm{\Omega}_{C^{1,\beta}}) \int_{\widetilde\Omega_R(x_0)} \lvert F\rvert^p \,\d x.
      \end{equation}
  \end{enumerate}
\end{lem}

\begin{proof}
  Fix a smooth domain $A \subset \bb R^n$ such that $\overline B_{\frac56}(0)^+ \subset A \subset B_1(0)^+.$ Using the graph representation above we can construct a diffeomorphism $\psi \colon B_{R_0}(x_0) \to U \subset B_1(0)$ such that $A \subset U,$ $\psi(B_{R_0} \cap \Omega) = U \cap \bb R^n_+,$ and such that $D\psi(x_0)$ is orthogonal. Hence by shrinking $R_0$ if necessary we can assume that
  \begin{equation}
    B_{\frac{5R}{6R_0}}(0) \subset \psi(B_R(x_0)) \subset B_{\frac{6R}{5R_0}}
  \end{equation}
for all $R \in (0,R_0).$ Hence if we let $\widetilde\Omega_R = \psi^{-1}\left(\frac{18R}{25R_0}A\right)$ this satisfies,
  \begin{equation}
    \overline{\Omega_{R/2}} \subset \psi^{-1}\left(\overline B_{\frac{3R}{5R_0}}(0)^+\right) \subset \widetilde\Omega_R \subset \psi^{-1}\left(B_{\frac{18R}{25R_0}}(0)^+\right) \subset \Omega_{R},
  \end{equation}
  as claimed. Now if $\varphi \in W^{1,2}(\widetilde\Omega_R,\bb R^N),$ setting $\widetilde\varphi = \varphi \circ \psi^{-1}$ we have for $\psi(y)=x$ that
  \begin{equation}
    -\div(\bb A \nabla \varphi) = - \div(\widetilde{\bb A} \nabla\widetilde\varphi)
  \end{equation}
  where we define
  \begin{equation}
    \widetilde{\bb A}(y)v : w = \lvert \det(\nabla\psi(y))\rvert^{-1}\, \bb A (\nabla\psi(x)v) : (\nabla \psi(x)w)
  \end{equation}
  for $y = \psi(x)$ and all $v,w \in \bb R^{Nn}.$ We can check $\widetilde{\bb A}$ is Legendre-Hadamard elliptic and $\beta$-H\"older continuous with constants depending on $n, \lambda, \Lambda$ and $\norm{\Omega}_{C^{1,\beta}},$ noting $\nabla\psi \in C^{0,\beta}$ with bounded inverse. Hence (i) and (ii) follow by analogous estimates on $A_R := \frac{18R}{25R_0} A$ applying the classical Schauder and Calder\'on-Zygmund estimates respectively; see for instance Theorems $10.12,$ $10.17$ in \cite{book:Giusti03} for details.
\end{proof}

\begin{rem}\label{rem:negative_sobolev_embedding}
  The second estimate (ii) replaces $W^{2,2}$ estimates by weaker bounds in $W^{1,p},$ which suffices for our application. We will apply this with $f \in L^2(\widetilde\Omega_R(x_0),\bb R^N)$ by using the Newtonian potential to define
\begin{equation}
  F = \frac{-1}{n\omega_n}\int_{\widetilde\Omega_R(x_0)} f(y) \frac{x-y}{\lvert x-y\rvert^{n}} \,\d x,
\end{equation}
which satisfies $-\div F = f\chi_{\widetilde\Omega_R(x_0)}$ in $\bb R^n.$ By standard potential estimates (see for instance Lemmas $7.12$, $7.14,$ and Theorem $9.9$ in \cite{book:GilbargTrudinger98}) we have $C=C(n,p)$ such that
\begin{equation}
  \norm{F}_{L^p(\widetilde\Omega_R(x_0))} \leq C \cal L^n\left(\widetilde\Omega_R(x_0)\right)^{\frac1n + \frac1p - \frac12} \norm{f}_{L^2(\widetilde\Omega_R(x_0))},
\end{equation}
provided $\frac12-\frac1p \leq \frac1n$ with $1 \leq p<\infty,$ which puts us in the setting of the above lemma.
\end{rem}

Finally we conclude by stating a Poincar\'e inequality we will use extensively later. For the case of the modified domain, this follows by flattening the boundary and rescaling the smooth domain $A$, whereas in $\Omega_R$ we can extend by zero to $B_R(x_0)$ and apply the corresponding inequality there.

\begin{lem}[Poincar\'e inequality]\label{lem:scaled_poincare}
  Let $\Omega \subset \bb R^n$ be a bounded $C^{1,\beta}$ domain and let $R_0>0,$ $\widetilde\Omega_R(x_0)$ as in Lemma \ref{lem:boundary_linearelliptic} above. Then for all $x_0 \in \partial\Omega,$ $0<R<R_0,$ $1<p<\infty,$ for $u \in W^{1,p}(\widetilde\Omega_R(x_0),\bb R^N)$ such that $u = 0$ on $\partial\Omega \cap B_R(x_0)$ in the trace sense we have
  \begin{equation}
    R^{\frac np-\frac nq - 1} \norm{u}_{L^q(\widetilde\Omega_R)} \leq C \norm{\nabla u}_{L^p(\widetilde\Omega_R)}
  \end{equation}
  for all $1 \leq q < \infty$ such that $\frac1p - \frac1q \leq \frac1n,$ with $C=C(n,p,q,\beta,\norm{\Omega}_{C^{1,\beta}})>0.$ Also the same conclusion holds for $\Omega_R(x_0)$ in place of $\widetilde\Omega_R(x_0).$
\end{lem}

\section{Regularity up to the boundary for \texorpdfstring{$F$}{F}-extremals}\label{sec:boundaryregularity}

We now use the results from the previous section to prove Theorem \ref{thm:autonomous_boundary_regularity}. The framework will be analogous to the interior regularity theory, involving establishing a Caccioppoli-type inequality and a harmonic approximation result.

We will continue to use the notation introduced in Section \ref{sec:autonomous_integrand}. Additionally, given a bounded $C^{1,\beta}$ domain $\Omega \subset \bb R^n,$ we will fix $R_0>0$ and $\delta \in (0,1)$ such that $\Omega_R(x_0)$ is a $\delta$-John domain for all $x_0 \in \partial\Omega,$ $0<R<R_0,$ and given $\rho$ as above we will also assume that we have $\cal L^n(\Omega_R(x_0)) \geq 4^{-n}\cal L^n(B_R(x_0))$ and
\begin{equation}
  C(n)^{-1} R^2 \leq \dashint_{\Omega_R(x_0)} \frac{\rho(x)^2}{\lvert \nabla\rho(x_0)\rvert^2} \,\d x \leq C(n) R^2,\label{eq:weightedcusp}
\end{equation}
for all $R<R_0.$ Shrinking $R_0$ further if necessary, we will moreover assume Proposition \ref{prop:local_john} and Lemmas \ref{lem:boundary_aconstruction}, \ref{lem:boundary_linearelliptic}, \ref{lem:scaled_poincare} from the previous section hold with this choice of $R_0.$

\subsection{Boundary Caccioppoli-type inequality}

\begin{lem}[Boundary Caccioppoli-type inequality]\label{lem:autonomous_boundary_caccioppoli}
  Suppose $F$ satisfies Hypotheses \ref{hyp:autonomous_F}, let $M\geq 1,$ and suppose $\Omega \subset \bb R^n$ is a bounded $C^{1,\beta}$ domain for some $\beta \in (0,1).$ Given $g \in C^{1,\beta}(\overline\Omega,\bb R^N),$ there is $R_0=R_0(n,\Omega)>0$ such that the following holds. Suppose $x_0 \in \partial\Omega,$ $0<R<R_0,$ and $u \in W^{1,q}_g(\Omega,\bb R^N)$ is $F$-extremal in $\Omega_R(x_0)$ such that $\nabla u \in \BMO(\Omega_R(x_0),\bb R^{Nn}),$ $\seminorm{\nabla u}_{\BMO(\Omega_R(x_0))}\leq 1,$ and $\lvert (\nabla u)_{\Omega_R}\rvert \leq M.$ Then if we define
  \begin{equation}\label{eq:autonomous_boundary_approximant}
    a_R(x) = \xi_R\, \frac{\rho(x)}{\lvert \nabla\rho(x_0)\rvert} = \frac{((u-g) \rho)_{\Omega_R}}{(\rho^2)_{\Omega_R}}\rho(x),
  \end{equation}
  with $\rho$ the defining function for $\Omega$ as in Section \ref{sec:boundary_averages}, we have the estimate
  \begin{equation}\label{eq:autonomous_boundary_caccioppoli}
    \begin{split}
      \dashint_{\Omega_{R/2}} \lvert \nabla u - (\nabla u)_{\Omega_{R/2}}\rvert^2 \,\d x &\leq C\,\gamma\left(\seminorm{\nabla u}_{\BMO(\Omega_R)}\right) \dashint_{\Omega_R} \lvert \nabla u - (\nabla u)_{\Omega_{R}}\rvert^2 \,\d x  \\
      &\quad + \frac{C}{R^2} \dashint_{\Omega_R} \lvert u - g - a_R\rvert^2 \,\d x + CM^{2(q-1)}R^{2\beta},
    \end{split}
  \end{equation}
  where setting $\widetilde M = C(n,\beta,\norm{\Omega}_{C^{1,\beta}},\norm{\nabla g}_{C^{0,\beta}(\Omega)})M,$ $\gamma \colon [0,\infty) \to [0,1]$ is a non-decreasing continuous function satisfying $\gamma(0)=0$ depending on $\omega_{\widetilde M}$ and $q$ only, and
  \begin{equation}
    C=C\left(n,N,q,K_{\widetilde M}/\lambda_{\widetilde M},\delta,\norm{\Omega}_{C^{1,\beta}},R_0,\seminorm{\nabla g}_{C^{0,\beta}(\Omega)}\right)>0.
  \end{equation}
\end{lem}

The main technical obstruction is that we need a suitable test function $\phi$ vanishing on $\partial\Omega \cap B_R(x_0)$ in our coercivity estimates. We will achieve this without flattening the boundary, using ideas from \textsc{Campos Cordero} \cite[Chapter 4]{thesis:Campos_Cordero} and results from Section \ref{sec:boundary_averages}.

\begin{rem}\label{rem:affine_subtraction}
  Similarly as in the interior case, the choice of $a_R(x)$ in (\ref{eq:autonomous_boundary_approximant}) ensures that
  \begin{equation}
    \xi \mapsto \int_{\Omega_R} \left\lvert u-g- \xi\, \frac{\rho}{\lvert \nabla\rho(x_0)\rvert}\right\rvert^2 \,\d x
  \end{equation}
  is minimised in $\xi \in \bb R^N$ when $\xi = \xi_R$ from \eqref{eq:autonomous_boundary_approximant}, as noted by \textsc{Kronz} \cite{article:Kronz05}. If we set $\xi = (\nabla(u-g) \cdot \nu(x_0))_{\Omega_R},$ these can be compared through estimate
  \begin{equation}
    \begin{split}
      &\lvert \xi_R -  (\nabla(u-g) \cdot \nu(x_0))_{\Omega_R}\rvert \\
      &\qquad \leq C \left(\dashint_{\Omega_R} \left\lvert \nabla(u - g) - (\nabla(u-g) \cdot \nu(x_0))_{\Omega_R} \tensor \frac{\nabla \rho}{\lvert \nabla\rho(x_0)\rvert}\right\rvert^2 \,\d x\right)^{\frac12} + CMR^{\beta},
    \end{split}
  \end{equation}
  where $C=C(n,\beta,\norm{\Omega}_{C^{1,\beta}})>0.$ This is proved in \cite[Lemma 2(ii)]{article:Kronz05}, relying on the Poincar\'e inequality (Lemma \ref{lem:scaled_poincare}) and (\ref{eq:weightedcusp}).
\end{rem}

\begin{proof}
  Let $R_0>0$ as in the beginning of this section, and define
  \begin{equation}
    w(x) = u(x) - g(x) - a_R(x),
  \end{equation}
  noting that $w = 0$ on $\partial\Omega \cap B_R(x_0).$ We also fix a cutoff $\eta \in C^{\infty}_c(B_R(x_0))$ such that $1_{B_{R/2}(x_0)} \leq \eta \leq 1_{B_R(x_0)}$ and $\lvert \nabla\eta\rvert \leq \frac CR,$ and consider the shifted functional $\widetilde F(z) = F_{z_{R}}(z)$ as in (\ref{eq:autonomous_shifted}) where
  \begin{equation}\label{eq:autonomous_z0_defn}
    z_{R} = \xi_R \tensor \nu_{x_0} + (\nabla g)_{\Omega_R},
  \end{equation}
  with $\xi_R$ as in (\ref{eq:autonomous_boundary_approximant}). Using the Poincar\'e inequality (Lemma \ref{lem:scaled_poincare}), we can choose $\widetilde M$ so that
  \begin{equation}
    \lvert z_R\rvert \leq C(n,\beta,\norm{\Omega}_{C^{1,\beta}}) \left(\dashint_{\Omega_R} \lvert \nabla u - \nabla g\rvert^2 \,\d x \right)^{\frac12} + \lvert (\nabla g)_{\Omega_R}\rvert \leq \widetilde M.
  \end{equation}
  Now by the strict Legendre-Hadamard condition applied to $\eta w$ and testing the equation (\ref{eq:autonomous_EL}) against $\eta^2w$ we have
  \begin{equation}
    \begin{split}
      \lambda_{\widetilde M} \dashint_{\Omega_R} \lvert \nabla(\eta w)\rvert^2 \,d x &\leq \dashint_{\Omega_R} \widetilde F''(0)\nabla(\eta w) : \nabla(\eta w) \,\d x - \dashint_{\Omega_R} \widetilde F'(\nabla u -z_{R}): \nabla(\eta^2w) \,\d x \\
      &= \dashint_{\Omega_R} \eta\left( \widetilde F''(0)(\nabla u - z_{R}) - \widetilde F'(\nabla u - z_{R})\right) : \nabla(\eta w) \,\d x \\
      &\quad+ \dashint_{\Omega_R} \eta \widetilde F''(0)(z_{R} - \nabla a_R - \nabla g) : \nabla(\eta w) \,\d x \\
      &\quad+ \dashint_{\Omega_R} w \widetilde F''(0)\nabla \eta : \nabla(\eta w) \,\d x -\dashint_{\Omega_R} \eta w \widetilde F'(\nabla u -z_{R}) : \nabla\eta \,\d x.
    \end{split}
  \end{equation}
  We can absorb the $\nabla(\eta w)$ terms using Cauchy-Schwarz and Young's inequality; for the last term we can use the growth estimate (\ref{eq:autonomous_shifted_estimate2}) for $\widetilde F'$ to estimate
  \begin{equation}
    \begin{split}
      &\dashint_{\Omega_R} \eta w \widetilde F'(\nabla u -z_{R}) : \nabla\eta \,\d x \\
      &\qquad\leq K_{\widetilde M}\dashint_{\Omega_R} \lvert w \nabla \eta\rvert \left( \lvert \eta(\nabla u - z_{R})\rvert + \eta\lvert \nabla u - z_{R}\rvert^{q-1}\right) \,\d x \\
      &\qquad \leq \frac{CK_{\widetilde M}^2}{\lambda_{\widetilde M}}\dashint_{\Omega_R} \lvert w \nabla \eta\rvert^2 \,\d x + \frac{\lambda_{\widetilde M}}8 \dashint_{\Omega_R} \lvert \nabla(\eta w)\rvert^2 + \eta^2 \lvert \nabla w\rvert^{2(q-1)} \,\d x \\
      &\qquad\quad + C\lambda_{\widetilde M}\dashint_{\Omega_R} \eta^2 \lvert z_R-a_R-\nabla g\rvert^2 + \eta^2 \lvert z_R-a_R-\nabla g\rvert^{2(q-1)} \,\d x.
    \end{split}
  \end{equation}
  Hence since $\eta^2 \leq 1$ we deduce that
  \begin{equation}
    \begin{split}
      \frac12 \dashint_{\Omega_R} \lvert \nabla(\eta w)\rvert^2 \,\d x &\leq \frac{4}{\lambda_{\widetilde M}} \dashint_{\Omega_R} \left\lvert  \widetilde F''(0)(\nabla u - z_{R}) - \widetilde F'(\nabla u - z_{R})\right\rvert^2 \,\d x \\
      &\quad + C\dashint_{\Omega_R} \left( \lvert z_{R} - \nabla a_R - \nabla g\rvert^2 + \lvert z_{R} - \nabla a_R - \nabla g\rvert^{2(q-1)} \right) \,\d x \\
      &\quad + \frac{C}{R^2} \dashint_{\Omega_R} \lvert w\rvert^2 \,\d x + C\dashint_{\Omega_R} \lvert \nabla w\rvert^{2(q-1)} \,\d x,
    \end{split}
  \end{equation}
  where the final term can be omitted if $q=2.$ For the second term we note that since $g, \rho$ are $C^{1,\beta}$ we have
  \begin{equation}
    \lvert z_{R} - \nabla a_R - \nabla g\rvert \leq \left\lvert  \xi_R \tensor \nu_{x_0}\right\rvert \frac{\lvert \nabla\rho(x)-\nabla\rho(x_0)\rvert}{\lvert \nabla\rho(x_0)\rvert} + \lvert \nabla g - (\nabla g)_{\Omega_R}\rvert \leq CM R^{\beta}
  \end{equation}
  in $\Omega_R,$ where $C = C(\norm{\Omega}_{C^{1,\beta}}, \seminorm{\nabla g}_{C^{0,\beta}})>0.$ For the first term we apply the comparison estimate (\ref{eq:autonomous_F_pertubation}); writing $\Phi(t) = \omega_{\widetilde M}(t)(t^2+t^{2(q-1)})$ this gives
  \begin{equation}
    \dashint_{\Omega_R} \left\lvert  \widetilde F''(0)(\nabla u - z_{R}) - \widetilde F'(\nabla u - z_{R})\right\rvert^2 \,\d x \leq K_{\widetilde M} \dashint_{\Omega_R} \Phi(\lvert \nabla u - z_{R}\rvert) \,\d x,
  \end{equation}
  noting that $\omega_{\widetilde M}(t)\leq 1.$ Now we estimate
  \begin{equation}\label{eq:autonomous_z0split}
    \begin{split}
      \lvert \nabla u - z_{R}\rvert &\leq \lvert \nabla u - (\nabla u)_{\Omega_R}\rvert+ \lvert \xi_R - ((\nabla(u - g)) \cdot \nu_{x_0})_{\Omega_R}\rvert \\
      &\quad+ \lvert (\nabla(u - g))_{\Omega_R} - ((\nabla( u - g)) \cdot \nu_{x_0})_{\Omega_R} \tensor \nu_{x_0}\rvert .
    \end{split}
  \end{equation}
  By Remark \ref{rem:affine_subtraction} the second term can be estimated as
  \begin{equation}
    \begin{split}
      &\lvert \xi_R - (\nabla(u-g)\cdot\nu_{x_0})_{\Omega_R}\rvert \\
      &\qquad\leq C\left(\int_{\Omega_R} \left\lvert \nabla u - \nabla g - (\nabla(u-g) \cdot \nu_{x_0})_{\Omega_R} \tensor \nu_{x_0}\right\rvert^2 \,\d x\right)^{\frac12} + CMR^{\beta} \\
      &\qquad\leq C \lvert (\nabla(u - g))_{\Omega_R} - ((\nabla( u - g)) \cdot \nu_{x_0})_{\Omega_R} \tensor \nu_{x_0}\rvert \\
      &\qquad\quad + C\left(\int_{\Omega_R} \left\lvert \nabla u - \nabla g - (\nabla(u-g))_{\Omega_R}\right\rvert^2 \,\d x\right)^{\frac12} + CMR^{\beta},
    \end{split}
  \end{equation}
  and applying \textsc{Campos Cordero}'s trick  (Lemma \ref{lem:boundary_aconstruction}) followed by the John-Nirenberg estimate (Proposition \ref{prop:global_JohnNirenberg}) we have
  \begin{equation}\label{eq:autonomous_remainder_averages}
    \begin{split}
      &\lvert (\nabla u - \nabla g)_{\Omega_R} - ((\nabla u - \nabla g) \cdot \nu_{x_0})_{\Omega_R} \tensor \nu_{x_0}\rvert \\
      &\qquad\leq C \left(\dashint_{\Omega_R} \lvert \nabla u - \nabla g - (\nabla u - \nabla g)_{\Omega_R}\rvert^p \,\d x\right)^{\frac1p} + CMR^{\beta} \\
      &\qquad\leq C\seminorm{\nabla u - \nabla g}_{\BMO(\Omega_R)} + CMR^{\beta},
    \end{split}
  \end{equation}
  for $p \in \{2,q\}.$ Also applying the modular Fefferman-Stein estimate (Corollary \ref{cor:BMO_modulus})  we can bound
  \begin{equation}
    \begin{split}
      &\dashint_{\Omega_R} \Phi\left(\lvert \nabla u - (\nabla u)_{\Omega_R}\rvert\right) \,\d x \\
      &\qquad\leq C\, \omega_{\widetilde M}\left(\seminorm{\nabla u}_{\BMO(\Omega_R)}\right) \dashint_{\Omega_R} \lvert \nabla u - (\nabla u)_{\Omega_R}\rvert^2 + \lvert \nabla u - (\nabla u)_{\Omega_R}\rvert^{2(q-1)} \,\d x.
    \end{split}
  \end{equation}
  Now since $\seminorm{\nabla g}_{\BMO(\Omega_R)} \leq CR^{\beta}$ and $\Phi(R^{\beta}) \leq \left(1+R_0^{2(q-2)}\right)R^{2\beta},$ we can combine the above using the doubling property of $\Phi$ to get
  \begin{equation}
    \begin{split}
      &\dashint_{\Omega_R} \Phi\left(\lvert \nabla u - z_0\rvert\right) \,\d x \leq  CM^{2(q-1)}R^{2\beta}\\
      &\qquad + C\, \omega_{\widetilde M}\left(\seminorm{\nabla u}_{\BMO(\Omega_R)}\right) \dashint_{\Omega_R} \lvert \nabla u - (\nabla u)_{\Omega_R}\rvert^2 + \lvert \nabla u - (\nabla u)_{\Omega_R}\rvert^{2(q-1)} \,\d x.
    \end{split}
 \end{equation}
 To complete the estimate, note by the John-Nirenberg inequality (Proposition \ref{prop:global_JohnNirenberg}) that
 \begin{equation}
   \dashint_{\Omega_R} \lvert \nabla u - (\nabla u)_{\Omega_R}\rvert^{2(q-1)} \leq C \seminorm{\nabla u}_{\BMO(\Omega_R)}^{2(q-2)} \dashint_{\Omega_R} \lvert \nabla u - (\nabla u)_{\Omega_R}\rvert^2 \,\d x,
 \end{equation}
 and similarly
 \begin{equation}
   \dashint_{\Omega_R} \lvert \nabla w\rvert^{2(q-1)} \,\d x \leq C\seminorm{\nabla u}_{\BMO(\Omega_R)}^{2(q-2)} \dashint_{\Omega_R} \lvert \nabla u - (\nabla u)_{\Omega_R}\rvert^2 \,\d x + CM^{2(q-1)}R^{2\beta}.
 \end{equation}
 Hence putting everything together gives
 \begin{equation}
   \begin{split}
     \dashint_{\Omega_R} \lvert \nabla(\eta w)\rvert^2 \,\d x &\leq C\,\omega_{\widetilde M}\left(\seminorm{\nabla u}_{\BMO(\Omega_R)}\right)\left(1 + \seminorm{\nabla u}_{\BMO(\Omega_R)}^{2(q-2)}\right) \dashint_{\Omega_R} \lvert \nabla u - (\nabla u)_{\Omega_R}\rvert^2 \\
     &\quad + \frac{C}{R^2}\dashint_{\Omega_R} \lvert w\rvert^2 \,\d x + C\seminorm{\nabla u}_{\BMO(\Omega_R)}^{2(q-2)} \dashint_{\Omega_R} \lvert \nabla u - (\nabla u)_{\Omega_R}\rvert^2 \,\d x \\
     &\quad+ CM^{2(q-1)}R^{2\beta},
   \end{split}
 \end{equation}
 from which the result follows taking $\gamma(t)=\min\{1,\omega_{\widetilde M}(t)(1+t^{2(q-2)})+t^{2(q-2)}\},$ omitting the $t^{2(q-2)}$ terms if $q=2.$
\end{proof}

\subsection{Boundary harmonic approximation}

\begin{lem}[Boundary harmonic approximation]\label{lem:autonomous_harmonic_boundary}
  Suppose $F$ satisfies Hypotheses \ref{hyp:autonomous_F}, let $M \geq 1,$ and suppose $\Omega \subset \bb R^n$ is a bounded $C^{1,\beta}$ domain and $g \in C^{1,\beta}(\overline\Omega,\bb R^N),$ for some $\beta \in (0,1).$ Suppose $x_0 \in\partial\Omega,$ $0<R <R_0$ with $R_0=R_0(n,\Omega)>0$ and $u \in W^{1,q}_g(\Omega_{R},\bb R^N)$ is $F$-extremal in $\Omega_R(x_0)$ with $\nabla u \in \BMO(\Omega_{R},\bb R^{Nn}),$ $\seminorm{\nabla u}_{\BMO(\Omega_R)} \leq 1,$ and $\lvert (\nabla u)_{\Omega_R}\rvert \leq M.$

  Then letting $\widetilde\Omega_R$ as in Lemma \ref{lem:boundary_linearelliptic}, the unique solution $h \in W^{1,2}(\widetilde\Omega_R,\bb R^N)$ to the Dirichlet problem
  \begin{equation}
    \pdeproblem{-\div F''(z_{R})\nabla h}{0}{\widetilde\Omega_R,}{h}{u-g-a_{R}}{\partial\widetilde\Omega_R,}
  \end{equation}
  with $z_{R}, a_{R}$ as in (\ref{eq:autonomous_boundary_approximant}), (\ref{eq:autonomous_z0_defn}) respectively satisfies
  \begin{equation}
    \int_{\widetilde\Omega_R} \lvert \nabla h\rvert^2 \,\d x \leq C \int_{\widetilde\Omega_R} \lvert \nabla(u-g-a_{R})\rvert^2 \,\d x,
  \end{equation}
  where $C=C\left(n,K_{\widetilde M}/\lambda_{\widetilde M}\right)>0$ with $\widetilde M = C\left(n,\beta,\norm{\Omega}_{C^{1,\beta}},\norm{\nabla g}_{C^{0,\beta}(\Omega)}\right)M.$ Moreover we have the remainder estimate
  \begin{equation}
    \frac1{R^2}\dashint_{\widetilde\Omega_R} \lvert u-g-a_{R}-h\rvert^2 \,\d x \leq C \gamma\left(\seminorm{\nabla u}_{\BMO(\Omega_{R})}\right) \dashint_{\Omega_{R}} \lvert \nabla u - (\nabla u)_{\Omega_{R}}\rvert^2 \,\d x + CM^2R^{2\beta},
  \end{equation}
  where $C=C\left(n,N,q,K_{\widetilde M}/\lambda_{\widetilde M},\norm{\Omega}_{C^{1,\beta}},\seminorm{\nabla g}_{C^{0,\beta}(\Omega)}\right)>0$ and $\gamma \colon [0,\infty) \to [0,1]$ non-decreasing continuous such that $\gamma(0)=0,$ depending on $n, q$ and $\omega_{\widetilde M}$ only.
\end{lem}

\begin{proof}
  We will assume $n\geq 3$ so Sobolev embedding applies, taking similar modifications as in the interior case if $n=2.$ Additionally we will use similar arguments used in the proof of Lemma \ref{lem:autonomous_boundary_caccioppoli} which we will not reproduce in detail, in particular choosing $R_0, \widetilde M$ in the same way. 
  As in the interior case we will also replace $F$ with $\lambda_{\widetilde M}^{-1}F.$
  Letting $\widetilde F = F_{z_{R}}$ be the shifted functional with $z_{R}$ as in (\ref{eq:autonomous_z0_defn}) and setting $w = u-g-a_{R},$ note for $\phi \in W^{1,2}_0(\widetilde\Omega_R,\bb R^N)$ we have
  \begin{equation}
    \begin{split}
      &\dashint_{\widetilde\Omega_R} \widetilde F''(0)(\nabla w - \nabla h) : \nabla\phi \,\d x \\
      &\qquad= \dashint_{\widetilde\Omega_R} \left(\widetilde F''(0)\nabla w - \widetilde F'(\nabla u - z_{R})\right) : \nabla \phi \,\d x\\
      &\qquad\leq K_{\widetilde M} \dashint_{\widetilde\Omega_R} \omega_M(\lvert \nabla u-z_{R}\rvert)\left(\lvert \nabla u - z_{R}\rvert+\lvert \nabla u-z_{R}\rvert^{q-1}\right) \lvert \nabla\phi\rvert \,\d x \\
      &\qquad\quad + K_{\widetilde M} \dashint_{\widetilde\Omega_R} \lvert \nabla a_{R} - \nabla g - z_{R}\rvert \lvert \nabla\phi\rvert \,\d x,
    \end{split}
  \end{equation}
  where we used the comparison estimate (\ref{eq:autonomous_F_pertubation}). We now choose $\phi$ to be the unique solution to the Dirichlet problem
  \begin{equation}
    \pdeproblem{-\div \widetilde F''(0)\nabla\phi}{w-h}{\widetilde\Omega_R,}{\phi}{0}{\partial\widetilde\Omega_R.}
  \end{equation}
  Since $w-h \in L^2(\widetilde\Omega_R) \hookrightarrow W^{-1,2^*}(\widetilde\Omega_R)$ by Remark \ref{rem:negative_sobolev_embedding}, by Lemma \ref{lem:boundary_linearelliptic}(ii) with $p=2^*$ we obtain the estimate $\norm{\nabla\phi}_{L^{2^*}(\widetilde\Omega_R)} \leq C \norm{w-h}_{L^2(\widetilde\Omega_R)}.$ Therefore for this choice of $\phi$ we get
  \begin{equation}
    \begin{split}
      \dashint_{\widetilde\Omega_R} \lvert w-h\rvert^2 \,\d x &\leq C\,\omega_{\widetilde M}\left(\dashint_{\Omega_{R}} \lvert \nabla u - z_{R}\rvert \,\d x\right)^{\frac2n} \dashint_{\Omega_{R}} \lvert \nabla u - z_{R}\rvert^2 + \lvert \nabla u - z_{R}\rvert^{2(q-1)} \,\d x \\
      &\quad + C \left( \dashint_{\Omega_{R}} \lvert \nabla a_{R} - \nabla g - z_{R}\rvert^{2_*} \,\d x \right)^{\frac2{2_*}},
    \end{split}
  \end{equation}
  where we have used H\"older and Jensen's inequalities (here $2_* = \frac{2n}{n+2}$), and absorbed the $\dashint_{\widetilde\Omega_R}\lvert w-h\rvert^2\,\d x$ term on the right-hand side. Arguing by splitting $\lvert \nabla u - z_{R}\rvert$ as in (\ref{eq:autonomous_z0split}) from the previous section (proof of Lemma \ref{lem:autonomous_boundary_caccioppoli}) we arrive at the estimate
  \begin{equation}
    \dashint_{\widetilde\Omega_R} \lvert w-h\rvert^2 \,\d x \leq C\gamma\left(\seminorm{\nabla u}_{\BMO(\Omega_{R})}\right) \dashint_{\Omega_{R}} \lvert \nabla u - (\nabla u)_{\Omega_{R}}\rvert^2 \,\d x + CM^2R^{2\beta}
  \end{equation}
  with $\gamma(t) = \min\{1,\omega_{\widetilde M}(t)^{\frac2n}(1+t^{2(q-2)})\},$ as required.
\end{proof}

\subsection{Boundary \texorpdfstring{$\eps$}{epsilon}-regularity and the controlled case}\label{sec:boundary_regularityproof}

We now combine the estimates from the previous sections to conclude as in the interior case.

\begin{proof}[Proof of Theorem {\ref{thm:autonomous_boundary_regularity}}]
  For $B_r(x) \subset B_{R_0}(x_0)$ with $x \in \overline\Omega$ we consider the excess energy
  \begin{equation}\label{eq:excess_energy}
    E(x,r) = \dashint_{\Omega_r(x)} \lvert \nabla u - (\nabla u)_{\Omega_r(x)}\rvert^2\,\d y,
  \end{equation}
  so by assumption and Proposition \ref{prop:global_JohnNirenberg} there is $C_1=C_1(n,\delta)>0$ such that $E(x,r) \leq C_1\eps^2,$ which we can assume is less than $1.$

  \textbf{Claim}: If $x \in \partial\Omega$ and $r>0$ so that $\Omega_r(x) \subset \Omega_R(x_0)$ and $\sigma \in \left(0,\frac14\right)$ for which
  \begin{equation}
    \lvert (\nabla u)_{\Omega_{2\sigma r}(x)}\rvert, \lvert (\nabla u)_{\Omega_r(x)}\rvert \leq 2^{3n+1}M,
  \end{equation}
  we have
  \begin{equation}\label{eq:boundary_excessdecay}
    E(x,\sigma r) \leq C \left( \sigma^{2\beta} + \sigma^{-(n+2)} \gamma\left(\seminorm{\nabla u}_{\BMO(\Omega_r(x))}\right)\right) E(x,r) + CM^{2(q-1)}\sigma^{-(n+2)}r^{2\beta},
  \end{equation}
  where $\gamma$ is as in Lemmas \ref{lem:autonomous_boundary_caccioppoli} and \ref{lem:autonomous_harmonic_boundary} with $2^{3n+1}M$ in place of $M,$ and
  \begin{equation}
    C=C\left(n,N,q,K_{\widetilde M}/\lambda_{\widetilde M},\delta, \norm{\Omega}_{C^{1,\beta}},R_0,\seminorm{\nabla g}_{C^{0,\beta}(\Omega)}\right)>0.
  \end{equation}

  \textit{Proof of claim}: Applying the Caccioppoli-type inequality (Lemma \ref{lem:autonomous_boundary_caccioppoli}) we have
  \begin{equation}
    \begin{split}
      E(x,\sigma r) &\leq C\,\gamma\left(\seminorm{\nabla u}_{\BMO(\Omega_{2\sigma r}(x))}\right) E(x,2\sigma r) \\
      &\quad+ \frac1{\sigma^2r^2} \dashint_{\Omega_{2\sigma r}(x)} \lvert u-g - a_{2\sigma r}\rvert^2 \,\d y + CM^{2(q-1)}(\sigma r)^{2\beta},
    \end{split}
  \end{equation}
  where $a_{2\sigma r}$ is given by (\ref{eq:autonomous_boundary_approximant}) in $\Omega_{2\sigma r}(x).$ Also by the boundary harmonic approximation (Lemma \ref{lem:autonomous_harmonic_boundary}) in $\Omega_{r}(x)$ the unique solution $h \in W^{1,2}(\widetilde\Omega_{r}(x),\bb R^N)$ solving
  \begin{equation}
    \pdeproblem{-\div F''(z_{r})\nabla h}{0}{\widetilde\Omega_{r}(x),}{h}{u-g-a_{r}}{\partial\widetilde\Omega_{r}(x),}
  \end{equation}
  satisfies
  \begin{equation}
    \frac1{\sigma^2r^2} \dashint_{\Omega_{r/2}(x)} \lvert u-g-a_{r}-h\rvert^2 \,\d y \leq C\,\gamma\left(\seminorm{\nabla u}_{\BMO(\Omega_r(x))}\right) E(x,r) + CM^2r^{2\beta},
  \end{equation}
  noting that $\Omega_{r/2}(x) \subset \widetilde\Omega_r(x) \subset \Omega_r(x).$ Now by Remark \ref{rem:affine_subtraction} we have
  \begin{equation}
    \frac1{\sigma^2r^2} \dashint_{\Omega_{2\sigma r}(x)} \left\lvert u-g - a_{\sigma}\right\rvert^2 \,\d y \leq \frac1{\sigma^2r^2} \dashint_{\Omega_{2\sigma r}(x)} \left\lvert  u-g - \xi \,\frac{\rho}{\lvert \nabla\rho(x)\rvert}\right\rvert^2 \,\d y
  \end{equation}
  for all $\xi \in \bb R^N,$ so taking $\xi = \xi_r + (\nabla h \cdot \nu_{x})_{\Omega_{2\sigma r}(x)}$ we can split
  \begin{equation}
    \begin{split}
      &\frac1{\sigma^2r^2} \dashint_{\Omega_{2\sigma r}(x)} \lvert u-g-a_{\sigma}\rvert^2 \,\d y \\
      &\qquad\leq \frac1{\sigma^2r^2} \dashint_{\Omega_{2\sigma r}(x)} \left\lvert h- (\nabla h \cdot \nu_{x})_{\Omega_{2\sigma r}(x)} \frac{\rho}{\lvert\nabla \rho(x)\rvert} \right\rvert^2 \,\d y \\
      &\qquad\quad+ C\sigma^{-(n+2)}  \gamma\left(\seminorm{\nabla u}_{\BMO(\Omega_r(x))}\right) E(x,r) + CM^{2(q-1)} \sigma^{-(n+2)} r^{2\beta}.
    \end{split}
  \end{equation}
  For the second term we use the Poincar\'e inequality (Lemma \ref{lem:scaled_poincare}) and Lemma \ref{lem:boundary_aconstruction} to estimate
  \begin{equation}
    \begin{split}
      &\frac1{\sigma^2r^2} \dashint_{\Omega_{2\sigma r}(x)} \left\lvert h - (\nabla h \cdot \nu_{x})_{\Omega_{2\sigma r}(x)} \frac{\rho}{\lvert \nabla\rho(x)\rvert}\right\rvert^2 \,\d y \\
      &\qquad\leq C \dashint_{\Omega_{2\sigma r}(x)} \left\lvert \nabla h - (\nabla h \cdot \nu_{x})_{\Omega_{2\sigma r}(x)} \frac{\nabla \rho}{\lvert \nabla\rho(x)\rvert}\right\rvert^2 \,\d y \\
      &\qquad\leq C \dashint_{\Omega_{2\sigma r}(x)} \left\lvert \nabla h - (\nabla h \cdot \nu_{x})_{\Omega_{2\sigma r}(x)} \tensor \nu_{x}\right\rvert^2 \,\d y + CM^2(\sigma r)^{2\beta} \\
      &\qquad\leq C\dashint_{\Omega_{2\sigma r}(x)} \left\lvert \nabla h - (\nabla h)_{\Omega_{2\sigma r}(x)}\right\rvert^2 \,\d y + CM^2\sigma^{-n}r^{2\beta},
    \end{split}
  \end{equation}
  where we have used the bound $\lvert (\nabla h)_{\Omega_{2\sigma r}(x)}\cdot \nu_x\rvert^2 \leq CM^2\sigma^{-n}.$ Now as $h$ vanishes on $\partial\Omega \cap \partial\Omega_r(x),$ using (\ref{eq:boundary_schauder}) from Lemma \ref{lem:boundary_linearelliptic}(ii) we have the estimate 
  \begin{equation}
    \seminorm{\nabla h}_{C^{0,\beta}(\overline\Omega_{r/2}(x))} \leq C\dashint_{\widetilde\Omega_r(x)} \lvert \nabla(u-g-a_R)\rvert \,\d y \leq CE(x,r) + CM^2r^{2\beta},
  \end{equation}
  where the last line is obtained by arguing as in the proof of Lemma \ref{lem:autonomous_boundary_caccioppoli}. Hence it follows that
  \begin{equation}
    \frac1{\sigma^2r^2} \dashint_{\Omega_{2\sigma r}(x)} \left\lvert h - (\nabla h \cdot \nu_{x})_{\Omega_{2\sigma r}(x)} \frac{\rho}{\lvert \nabla\rho(x)\rvert}\right\rvert^2 \,\d y \leq C\sigma^{2\beta}E(x,r) + CM^2\sigma^{-n}r^{2\beta},
  \end{equation}
  so the claim follows by putting everything together.

  We now argue analogously as in the interior case; note for $x \in \partial\Omega \cap B_{R/2}(x_0)$ we have $\lvert (\nabla u)_{\Omega_{R/2}(x)}\rvert \leq 2^{3n}M,$ and so $\lvert (\nabla u)_{\Omega_{\sigma R/2}(x)}\rvert \leq 2^{3n}M+C_1\sigma^{-n}\eps \leq 2^{3n+1}M$ for $\eps>0$ sufficiently small. Hence applying the claim gives
  \begin{equation}
    E(x,\sigma R/2) \leq C\left(\sigma^{2\beta} + \sigma^{-(n+2)}\gamma(\eps)\right)E(x,r/2) + CM^{2(q-1)}\sigma^{-(n+2)}\widetilde R_0^{2(\beta-\alpha)}R^{2\alpha}.
  \end{equation}
  We choose $\sigma \in \left(0,\frac14\right)$ such that $C\sigma^{2\beta} \leq \frac14 \sigma^{2\alpha},$ and $\eps>0$ such that $C\sigma^{-(n+2)}\gamma(\eps) \leq \frac14 \sigma^{2\alpha}.$ We then choose $\widetilde R_0>0$ such that $CM^{2(q-1)}\sigma^{-(n+2)}\widetilde R_0^{2(\beta-\alpha)} \leq \kappa\sigma^{2\alpha}$ for $0<\kappa<1$ to be chosen to get
  \begin{equation}
    E(x,\sigma R/2) \leq \frac12 \sigma^{2\alpha} E(x,R/2) + \kappa\left(\sigma R\right)^{2\alpha}.
  \end{equation}
  Further shrinking $\eps>0$ if necessary and taking $\kappa>0$ small enough so
  \begin{equation}
    \sigma^{-(n+2)}(C_1\eps+\kappa) \sum_j \sigma^{\alpha j} \leq 3^nM,
  \end{equation}
  we can iteratively argue that for all $k \geq 0,$
  \begin{align}
    \lvert (\nabla u)_{B_{\sigma^kR/2}(x)}\rvert &\leq 2^{3n+1}M, \\
    E(x,\sigma^kR/2) &\leq 2^{-k}\sigma^{2\alpha k}E(x,R/2) + (\sigma^kR)^{2\alpha}.
  \end{align}
  Hence it follows that $E(x,r) \leq Cr^{2\alpha}$ for all $r \in (0,R/2).$ 

  By the interior case we also have $E(x,r) \leq Cr^{2\alpha}$ when $B(x,r) \subset \Omega_{R}(x_0)$ with $x \in B_{R/2}.$ We can extend this to all $x \in \Omega_{R/2}(x_0)$ and $0 < r < R/2$ by a covering argument (adjusting constants as necessary), so by the Campanato-Meyers characterisation we get $u$ is $C^{1,\alpha}$ in $\overline\Omega_{R/2}(x_0),$ as required.
\end{proof}

We now turn to the proof of Theorem \ref{thm:controlled_almostvmo}.
The key point is the follow lemma, which asserts that we obtain estimates analogous to those established in Section \ref{sec:autonomous_integrand}, with a precise dependence on $\lvert w\rvert \leq M$.

\begin{lem}\label{lem:controlled_estimates}
  Suppose $F$ satisfies Hypotheses \ref{hyp:controlled_F} for some $p \geq 2.$ Then there is $K>0$ such that for any $z,w \in \bb R^{Nn}$ we have \eqref{eq:autonomous_shifted} satisfies
\begin{align}
  \lvert F_w(z)\rvert &\leq K(1+\lvert w\rvert)^{p-2}( \lvert z\rvert^2 + \lvert z\rvert^p), \label{eq:controlled_shifted_estimate1}\\
  \lvert F'_w(z)\rvert &\leq K(1+\lvert w\rvert)^{p-2}( \lvert z\rvert + \lvert z\rvert^{p-1}),\label{eq:controlled_shifted_estimate2} \\
  \lvert F''_w(0)\rvert &\leq K(1+\lvert w\rvert)^{p-2}, \label{eq:controlled_shifted_twice}
\end{align}
and
\begin{equation}\label{eq:controlled_F_pertubation}
  \lvert F''_w(0)z - F'_w(z)\rvert \leq K(1+\lvert w\rvert)^{p-2}\omega(\lvert z\rvert)(\lvert z\rvert+\lvert z\rvert^{p-1}).
\end{equation}
for all $z,w \in \bb R^{Nn},$ with $\omega \colon [0,\infty) \to [0,1]$ a non-decreasing, continuous, and concave function such that $\omega(0)=0.$.
\end{lem}

\begin{proof}
  Quantifying \ref{item:controlled_growth} we have $F''(z)/(1+\lvert z\rvert)^{p-2}$ is bounded by $K$, and we let $\omega$ denote the associated modulus of continuity.
  From this \eqref{eq:controlled_shifted_twice} immediately follows, as does \eqref{eq:controlled_shifted_estimate1}, \eqref{eq:controlled_shifted_estimate2} by noting that
  \begin{equation}
    \lvert F_w(z) \rvert \leq \twopartdef{CK (1+\lvert w\rvert)^{p-2} \lvert z\rvert^2}{\text{ if } \lvert z\rvert\leq 1,}{\Lambda (1 + \lvert z\rvert)^p}{\text{ if } \lvert z\rvert > 1,}
  \end{equation} 
  and similarly for $F_w'(z).$
  Also if $\lvert z\rvert \leq 1$ we have
  \begin{equation}
    \begin{split}
      \lvert F''(w+z) - F''(w)\rvert &\leq (1+\lvert w+z\rvert)^{p-2} \left\lvert \frac{F''(w+z)}{(1+\lvert w+z\rvert)^{p-2}} - \frac{F''(w)}{(1+\lvert w\rvert)^{p-2}}\right\rvert \\
                                     &\quad+ \frac{\lvert F''(w)\rvert}{(1+\lvert w\rvert)^{p-2}} \left\lvert (1+\lvert w\rvert)^{p-2}- (1+\lvert w+z\rvert)^{p-2} \right\rvert \\
                                     &\leq (2+ \lvert w\rvert)^{p-2} K \omega(\lvert z\rvert) + CK (1 + \lvert w\rvert)^{p-2} \lvert z- w\rvert,
    \end{split}
  \end{equation} 
  where the second term is estimated by distinguishing between the cases $p \in [2,3]$ and $p > 3$.
  Hence taking $\widetilde\omega = \min\{1, \omega(t) + t\}$ we deduce that
  \begin{equation}
    \lvert F_w''(0)z - F_w'(z)\rvert \leq C K(1+\lvert w\rvert)^{p-2} \widetilde\omega(\lvert z\rvert) \lvert z\rvert
  \end{equation} 
  for $\lvert z \rvert \leq 1,$ and when $\lvert z\rvert \geq 1$ we use \eqref{eq:controlled_shifted_estimate2}, \eqref{eq:controlled_shifted_twice} to estimate
  \begin{equation}
    \lvert F_w''(0)z - F_w'(z)\rvert \leq C K(1+\lvert w\rvert)^{p-2} ( \lvert z\rvert + \lvert z\rvert^{p-1}),
  \end{equation} 
  so combining these \eqref{eq:controlled_F_pertubation} follows, replacing $CK, \widetilde\omega$ by $K, \omega$ respectively.
\end{proof}

\begin{proof}[Proof of Theorem \ref{thm:controlled_almostvmo}]
  Owing to Lemma \ref{lem:controlled_estimates}, the constants $K_M, \lambda_M$ from Section \ref{sec:autonomous_integrand} can be chosen so that $K_M / \lambda_M$ is independent of $M \geq \lvert z_0\rvert.$
  Similarly, we have the modulus of continuity $\omega = \omega_M$ is also independent of $M.$
  Hence we claim the following excess decay estimate
  \begin{equation}
    \begin{split}
      E(x,\sigma r) &\leq C\left(\sigma^{2\beta}+ \sigma^{-(n+2)} \gamma\left( \seminorm{\nabla u}_{\BMO(\Omega_r(x))}\right)\right) E(x,r) \\
      &\quad + C(1+\lvert (\nabla u)_{\Omega_{2\sigma r}}\rvert + \lvert (\nabla u)_{\Omega_{r}}\rvert)^{2}\sigma^{-(n+2)}r^{2\beta}
    \end{split}
  \end{equation}
  holds for all $x \in \overline\Omega,$ $R>0$ such that either $B_R(x) \subset \Omega$ or $x \in \overline\Omega$ and $0<R<R_0$ (with $R_0=R_0(n,\Omega)>0$).
  Indeed this follows from the excess decay estimates \eqref{eq:autonomous_excessdecay}, \eqref{eq:boundary_excessdecay} from the proofs of Theorems \ref{thm:autonomous_regularity} and \ref{thm:autonomous_boundary_regularity} respectively.
  Letting $M>0$ such that $\lvert (\nabla u)_{\Omega_{2\sigma r}}\rvert + \lvert (\nabla u)_{\Omega_{r}}\rvert \leq CM,$ in the above estimates we have $C$ and $\gamma$ depends on $M$ only through $K_M/\lambda_M$ and $\omega_M$, hence under our assumptions they are independent of $M$.
  Note in the interior case the second term can be omitted.

  Fix $\eps>0$ to be determined. Then there is $0<R<\frac{R_0}2$ for which there exists a finite covering of $\Omega$ by balls $\{B_R(x_j)\}$ where either $B_R(x_j) \subset \Omega$ or $x_j \in \overline\Omega,$ and $\seminorm{\nabla u}_{\BMO(\Omega_{2R}(x_j))} \leq 2\eps \leq 1$ for each $j.$ Let $M>0$ such that $\lvert (\nabla u)_{\Omega_{2R}(x_j)}\rvert \leq M$ for all $j,$ then observe that for all $x \in \overline\Omega$ and $0<r<R$ we have $\lvert (\nabla u)_{\Omega_r(x)}\rvert \leq C(n)M\left(1+ \log(R/r)\right).$ Hence the excess decay estimate becomes
  \begin{equation}
    E(x,\sigma r) \leq C\left(\sigma^{2\beta} + \sigma^{-(n+2)}\gamma(2\eps)\right)E(x,r) + CM^2 \sigma^{-(2n+2)} r^{2\beta} \left(1+\log(R/r)\right)
  \end{equation}
  whenever $0<r<R,$ and modifying constants this holds for all $x \in \overline\Omega.$

  Now choose $\sigma \in (0,\frac14)$ such that $C\sigma^{2\beta} \leq \frac14 \sigma^{2\alpha},$ and $\eps>0$ such that $C\sigma^{-(n+2)}\gamma(2\eps) \leq \frac14 \sigma^{2\alpha}.$ Then choose $0<r_0<R$ such that $CM^2\sigma^{-(2n+2)}r_0^{2(\beta-\alpha)} \left(1+\log(R/r_0)\right) \leq \sigma^{2\alpha}.$ This gives
  \begin{equation}
    E(x,\sigma r) \leq \frac12\sigma^{2\alpha}E(x,r) + (\sigma r)^{2\alpha},
  \end{equation}
  from which the result follows by iteration as in the proof of Theorem \ref{thm:autonomous_boundary_regularity}. 
\end{proof}

\section{Extensions}\label{sec:extensions}

Up until now we have confined our discussion to the setting of autonomous integrands, however the framework we developed extends to more general elliptic systems and higher order equation.
Rather than state the most general case possible, we will aim to highlight the necessary changes to adapt our arguments to these more general situations.

\subsection{Quasilinear elliptic systems}\label{sec:general_elliptic}

While our motivation for this investigation arose from studying the behaviour of extremals, it turns out our arguments do not make use of the variational structure of the equation. We will illustrate this by considering general Legendre-Hadamard elliptic systems, and also show how lower order terms can be handled.

More precisely we consider weak solutions to the equation
\begin{equation}\label{eq:elliptic_equation}
  -\div A(x,u,\nabla u) + B(x,u,\nabla u) = 0
\end{equation}
in $\Omega,$ subject to the following conditions.

\begin{hyp}\label{hyp:elliptic_coefficients}
  Let $n \geq 2,$ $N\geq 1,$ $\beta \in (0,1),$ $q\geq 2$ and $\Omega \subset \bb R^n$ a bounded $C^{1,\beta}$ domain. We consider Carath\'eodory functions
  \begin{align}
    A \colon \overline\Omega \times \bb R^N \times \bb R^{Nn} & \to \bb R^{Nn}, \\
    B \colon \overline\Omega \times \bb R^N \times \bb R^{Nn} & \to \bb R^{N},
  \end{align}
  satisfying the following (we use $D_u, D_z$ to denote partial derivatives in $u,$ $z$ respectively).
  \begin{enumerate}[label=(A\arabic*)]

    \item For all $(x,u,z) \in \overline\Omega \times \bb R^N \times \bb R^{Nn}$ we have
      \begin{equation*}\label{eq:elliptic_qgrowth}
        \lvert A(x,u,z)\rvert + \lvert B(x,u,z)\rvert \leq K(1+\lvert z\rvert^{q-1}).
      \end{equation*}

    \item The map $z \mapsto A(x,u,z)$ is continuously differentiable for each $(x,u),$ and for all $M>0$ there is $\Lambda_M>0$ and a continuous, non-decreasing concave function $\omega_M \colon [0,\infty) \to [0,1]$ satisfying $\omega_M(0)=0$ such that
      \begin{equation*}\label{eq:elliptic_zcontinuity}
        \lvert D_zA(x,u,z_1) - D_zA(x,u,z_2)\rvert \leq \Lambda_M \omega_M(\lvert z_1-z_2\rvert)
      \end{equation*}
      for all $x \in \overline\Omega,$ $\lvert u\rvert \leq M$ and $\lvert z_1\rvert,\lvert z_2\rvert \leq M+1.$

    \item For all $M>0,$ for $x \in \overline\Omega$ and $\lvert u\rvert,\lvert z\rvert \leq M$ we have the strong Legendre-Hadamard ellipticity condition
      \begin{equation*}\label{eq:elliptic_LH}
        D_zA(x,u,z)(\xi \tensor \eta) : (\xi \tensor \eta) \geq \lambda_M \lvert \xi\rvert^2 \lvert \eta\rvert^2
      \end{equation*}
      for all $\xi \in \bb R^N$ and $\eta \in \bb R^n.$

    \item For all $x_1,x_2 \in \overline\Omega,$ $u_1,u_2 \in \bb R^N$ and $z \in \bb R^{Nn}$ we have
      \begin{equation*}\label{eq:elliptic_xucontinuity}
        \lvert A(x_1,u_1,z) - A(x_2,u_2,z)\rvert \leq K(1+\lvert z\rvert^{q-1}) \varrho_{\beta}(\lvert x_1-x_2\rvert + \lvert u_1-u_2\rvert),
      \end{equation*}
      where $\varrho_{\beta}(t) = \min\{1,t^{\beta}\}.$

  \end{enumerate}

\end{hyp}

\begin{rem}
  A special case of the above is the Euler-Lagrange system associated to the non-autonomous integrand $F=F(x,u,z).$ Here the Euler-Lagrange system reads
  \begin{equation}
    - \div D_zF(x,u,\nabla u) + D_uF(x,u,\nabla u) = 0,
  \end{equation}
  so we need $F$ to be $C^2$ in $z$ and $C^1$ in $x,$ such that Hypotheses \ref{hyp:elliptic_coefficients} are satisfied with $A(x,u,z)=D_zF(x,u,z)$ and $B(x,u,z)=D_uF(x,u,z).$
\end{rem}

\begin{thm}[$\BMO$ $\eps$-regularity theorem for elliptic systems]\label{thm:elliptic_epsreg}
  Suppose $\Omega, A,B$ satisfies Hypotheses \ref{hyp:elliptic_coefficients} and suppose $u \in W^{1,q}_g(\Omega,\bb R^N)$ solves (\ref{eq:elliptic_equation}) with $g \in C^{1,\beta}(\overline\Omega,\bb R^N).$ Then for each $\alpha \in (0,\beta)$ and $M>0$ there is $\eps>0$ and $\widetilde R_0>0$ such that if $x \in \overline\Omega$ and $R \in (0,\widetilde R_0)$ such that $\lvert (\nabla u)_{\Omega_R(x_0)}\rvert \leq M$ and
  \begin{equation}
    \seminorm{\nabla u}_{\BMO(\Omega_R(x_0))} \leq \eps,
  \end{equation}
  then $u$ is $C^{1,\alpha}$ in $\overline{\Omega_{R/2}(x_0)}.$
\end{thm}

\noindent\textbf{Step 0: Reduction and linearisation}:
Our strategy will be similar to before; we fix $x_0 \in \overline\Omega$ and $R>0$ such that either $B_R(x_0) \subset \Omega,$ or $x_0 \in \partial\Omega$ and $0<R<R_0$ with $R_0>0$ as in the start of Section \ref{sec:boundaryregularity}. We will focus our attention to the boundary case, as the interior case is similar but simpler.
We will also fix $M>0$ such that $\lvert (\nabla u)_{\Omega_R(x_0)}\rvert \leq M$. 

We first observe that we can suppress the $u$-dependence; since $\nabla u \in \BMO(\Omega,\bb R^{Nn})$ we can use the John-Nirenberg and Sobolev inequalities to obtain $u \in C^{0,\chi}(\overline\Omega,\bb R^N)$ for all $\chi \in (0,1).$ Then fixing any $\widetilde\beta \in (\alpha,\beta)$ and taking $\chi = \widetilde\beta/\beta,$ we see that $x \mapsto A(x,u(x),z)$ and $x \mapsto B(x,u(x),z)$ are $\widetilde\beta$-H\"older continuous in $\overline\Omega.$ Hence changing the constant $K$ in \ref{eq:elliptic_xucontinuity} (depending on $n,$ $\Omega,$ $M$) we can assume $A,$ $B$ are independent of $u.$

We then consider the linearisation
\begin{equation}
  \widetilde A(z) = A(x_0,z+z_0))-A(x_0,z_0),
\end{equation}
which satisfies the growth estimates
\begin{align}
    \widetilde A(z) &\leq K_M (\lvert z\rvert + \lvert z\rvert^{q-1})\label{eq:quasilinear_growth}\\
    \widetilde A'(0) &\leq K_M \label{eq:quasilinear_second}\\
    \lvert \widetilde A'(0)z - \widetilde A(z)\rvert &\leq K_M\,\omega_{\widetilde M}(\lvert z\rvert)(\lvert z\rvert+\lvert z\rvert^{q-1}) \label{eq:quasilinear_pertubation}
\end{align}
for all $z \in \bb R^{Nn}$, along with the coercivity estimate
\begin{equation}\label{eq:quasilinear_lh}
  \lambda_M \dashint_{\Omega_R} \lvert \nabla\phi\rvert^2 \,\d x \leq \dashint_{\Omega_R} \widetilde A'(0)\nabla\phi : \nabla \phi \,\d x
\end{equation}
for all $\phi \in W^{1,2}_0(\Omega,\bb R^N).$ 

From here one can proceed analogously as in the autonomous case detailed in Sections \ref{sec:interiorregularity} and \ref{sec:boundaryregularity} replacing $\widetilde F'$ with $\widetilde A.$
We will sketch how the details can be modified, however the only difference is that we obtain extra terms arising from the $x$-dependence and the presence of the lower order term $B.$

\textbf{Step 1: Caccioppoli inequality}:
We claim that
  \begin{equation}\label{eq:elliptic_caccioppoli}
    \begin{split}
      \dashint_{\Omega_{R/2}} \lvert\nabla u - (\nabla u)_{\Omega_{R/2}}\rvert^2 \,\d x &\leq C\,\gamma\left(\seminorm{\nabla u}_{\BMO(\Omega_R)}\right) \dashint_{\Omega_R} \lvert\nabla u - (\nabla u)_{\Omega_{R}}\rvert^2 \,\d x  \\
      &\quad + \frac{C}{R^2} \dashint_{\Omega_R} \lvert u - g - a_R\rvert^2 \,\d x + CR^{2\beta},
    \end{split}
  \end{equation}
  with $\gamma(t) = \min\{1, \omega_{\widetilde M}(t)(1+t^{2(q-2)}) + t^{2(q-2)}\}$, omitting the $t^{2(q-2)}$ terms if $q =2$.
To show this, as before we will fix a cutoff $\eta \in C^{\infty}_c(B_R)$ satisfying $1_{B_{R/2}} \leq \eta \leq 1_{B_R},$ $\lvert \nabla \eta\rvert \leq \frac CR.$ 
Taking a $a_R$ as in \eqref{eq:autonomous_boundary_approximant} and set $z_R = \nabla a_R(x_0) + (\nabla g)_{\Omega_R}$. We then consider the linearisation $\widetilde A(z)$ with this choice of $z_R$, and also put $w=u-g-a_R.$
By the Legendre-Hadamard condition \eqref{eq:quasilinear_lh} we have
  \begin{equation}
    \lambda_M \dashint_{\Omega_R} \lvert \nabla(\eta w)\rvert^2 \,\d x \leq \dashint_{\Omega_R} \widetilde A'(0)\nabla(\eta w) : \nabla(\eta w) \,\d x,
  \end{equation}
  and since $u$ weakly solves (\ref{eq:elliptic_equation}) we have
  \begin{equation}
    0 = \dashint_{\Omega_R} A_z(x,\nabla u) : \nabla(\eta^2w) + B(x,\nabla u) (\eta^2 w) \,\d x,
  \end{equation}
  so combining these estimates we obtain
  \begin{equation}
    \begin{split}
      \lambda_M \dashint_{\Omega_R} \lvert \nabla(\eta w)\rvert^2 \,\d x &\leq \dashint_{\Omega_R} \eta\left(\widetilde A'(0)(\nabla u - z_R) - \widetilde A(\nabla u - z_R)\right) : \nabla(\eta w) \,\d x \\
      &\quad + \dashint_{\Omega_R} \eta\, \widetilde A'(0)(z_R-\nabla a_R - \nabla g) : \nabla(\eta w) \,\d x \\
      &\quad + \dashint_{\Omega_R} \left(A(x,\nabla u) - \widetilde A(\nabla u - z_R)\right) : \nabla(\eta w) \,\d x, \\
      &\quad + \dashint_{\Omega_R} w\, \widetilde A'(0)\nabla \eta : \nabla(\eta w) - \eta w \widetilde A(\nabla u - z_R) : \nabla \eta \\
      &\quad - \dashint_{\Omega} \eta \, B(x,\nabla u) \cdot \eta w \,\d x \\
      &= I + I\!I + I\!I\!I + I\!V + V.
    \end{split}
  \end{equation}
  We can argue exactly as in the autonomous case (proof of Lemma \ref{lem:autonomous_boundary_caccioppoli}) to estimate the terms $I, I\!I, I\!V$ as before using \eqref{eq:quasilinear_growth}, \eqref{eq:quasilinear_second}, \eqref{eq:quasilinear_pertubation}. For the remaining terms note that
  \begin{equation}
    \begin{split}
      I\!I\!I &= \dashint_{\Omega_R} \left( A(x,\nabla u) - A(x_0,\nabla u)\right) : \nabla(\eta w) \,\d x\\
      &\leq CR^{\beta}\dashint_{\Omega_R} \left(1+\lvert \nabla u-z_R\rvert^{q-1}\right) \lvert \nabla(\eta w)\rvert \\
      &\leq \frac{\lambda_M}{16} \dashint_{\Omega_R} \lvert \nabla(\eta w)\rvert^2 \,\d x + C\left(1+\seminorm{\nabla u}_{\BMO(\Omega_R)}^{2(q-1)} + R^{2\beta}\right)R^{2\beta}
    \end{split}
  \end{equation}
  where we have used the fact that $\dashint_{\Omega_R} A(x_0,z_R) : \nabla(\eta w) \,\d x = 0$ and \ref{eq:elliptic_xucontinuity} in the second line, and the last line follows from similar bounds given in the proof of Lemma \ref{lem:autonomous_boundary_caccioppoli}. Finally for the last term we can estimate
  \begin{equation}
    \begin{split}
      V &\leq K\dashint_{\Omega} (1 + \lvert \nabla u-z_R\rvert^{q-1}) \lvert \eta w\rvert \\
      &\leq CR \left(1 + \seminorm{\nabla u}_{\BMO(\Omega_R)}^{2(q-1)} + R^{2\beta} \right)^{\frac12} \left( \frac1{R^2}\dashint_{\Omega_R} \lvert \eta w\rvert^2 \,\d x \right)^{\frac12}\\
      &\leq CR^2 + \frac{\lambda_M}{16}\dashint_{\Omega_R} \lvert \nabla(\eta w)\rvert^2 \,\d x,
    \end{split}
  \end{equation}
  where we have used the Poincar\'e inequality (Lemma \ref{lem:scaled_poincare}) in the last line, which allows us to absorb the $\nabla(\eta w)$ term. Hence the result follows by putting everything together. 

\noindent\textbf{Step 2: Harmonic approximation}:
We now introduce the harmonic approximation which solves
\begin{equation}
  \pdeproblem{-\div \widetilde A'(0)\nabla h}{0}{\widetilde\Omega_R,}{h}{w}{\partial\widetilde\Omega_R,}
\end{equation}
along with the dual problem
\begin{equation}
  \pdeproblem{-\div \widetilde A'(0)\nabla \phi}{w-h}{\widetilde\Omega_R,}{\phi}{0}{\partial\widetilde\Omega_R,}
\end{equation} 
which lies in $W^{1,2^*}_0(\widetilde\Omega_R,\bb R^N).$
Using $\phi$ as a test function we obtain
  \begin{equation}
    \begin{split}
      &\dashint_{\widetilde\Omega_R} \lvert w-h\rvert^2 \,\d x  = \dashint_{\widetilde\Omega_R} \widetilde A'(0)(\nabla w - \nabla h) : \phi - A(x, \nabla u) : \nabla \phi - B(x,\nabla u) \cdot \phi \,\d x \\
      &\qquad\leq \dashint_{\widetilde\Omega_R} \left(\widetilde A'(0)(\nabla u - z_R) - \widetilde A(\nabla u - z_R)\right) : \nabla \phi \,\d x \\
      &\qquad\quad+ \dashint_{\Omega_R} \widetilde A'(0)(z_R - \nabla a_R - \nabla g) : \nabla\phi \,\d x\\
      &\qquad\quad + \dashint_{\Omega_R} \left(A(x,\nabla u) - \widetilde A(\nabla u - z_R)\right) : \nabla\phi \,\d x - \dashint_{\Omega_R} B(x,\nabla u) \cdot \phi \,\d x.
    \end{split}
  \end{equation}
  The first two terms can be estimated as in Lemma \ref{lem:autonomous_harmonic_boundary}, and for the latter two terms we have (making suitable modifications if $n=2$),
  \begin{align}
    &\dashint_{\Omega_R} \left(A(x,\nabla u) - \widetilde A(\nabla u - z_R)\right) : \nabla\phi \,\d x \leq CR^{\beta} \dashint_{\widetilde\Omega_R} (1+\lvert \nabla u -z_R\rvert^{(q-1)})\lvert \nabla\phi\rvert \,\d x,\\
    &\dashint_{\Omega_R} B(x,\nabla u) : \phi \,\d x \leq CR \left(\dashint_{\widetilde\Omega_R} 1+\lvert \nabla u-z_R\rvert^{2_*(q-1)} \,\d x\right)^{\frac1{2_*}} \left( \frac1{R^{2^*}} \dashint_{\Omega_R} \lvert \phi\rvert^{2^*} \,\d x\right)^{\frac1{2^*}},
  \end{align}
  which can be controlled similarly as the previous step using  along with the Poincar\'e inequality (Lemma \ref{lem:scaled_poincare}) with $\phi$ for the second term.
  Therefore we obtain the remainder estimate
  \begin{equation}\label{eq:quasilinear_harmonic}
    \frac1{R^2}\dashint_{\widetilde\Omega_R} \lvert w-h\rvert^2 \,\d x \leq C \,\gamma\left(\seminorm{\nabla u}_{\BMO(\Omega_{R})}\right) \dashint_{\Omega_{R}} \lvert \nabla u - (\nabla u)_{\Omega_{R}}\rvert^2 \,\d x + CR^{2\beta},
  \end{equation}
  where $\gamma(t) = \min\{1,\omega_{\widetilde M}(t)^{\frac2n}(1+t^{2(q-2)})\}.$

  \noindent\textbf{Step 3: Excess decay and conclusion}: 
  Now we can combine the above two estimates to deduce decay estimates for the excess energy \eqref{eq:excess_energy}.
  Since the estimates \eqref{eq:elliptic_caccioppoli} and \eqref{eq:quasilinear_harmonic} are identical to the estimates established in Lemmas \ref{lem:autonomous_boundary_caccioppoli}, \ref{lem:autonomous_harmonic_boundary}, we can argue exactly as in Section \ref{sec:boundary_regularityproof} to conclude.
  Thus we have established Theorem \ref{thm:elliptic_epsreg}.

\subsection{Higher order integrands}\label{sec:higherorder}

We will also outline how analogous results can be obtained for $k$\th order problems. For this fix $k \geq 1,$ and let $\bb M_k = \sym_k(\bb R^n,\bb R^N)$ denote the space of symmetric $k$-linear maps $(\bb R^{n})^k \to \bb R^n.$ If $\xi \in \bb R^N$ and $\eta \in \bb R^n,$ we write $\eta^k = \eta \, \tensor \cdots \tensor\, \eta$ to denote the $k$-fold tensor product and identify elements $\xi \tensor \eta^{k} \in \bb M_k$ to send $(x_1,\dots,x_k) \to \xi\sum_{\lvert \alpha\rvert=k} x^{\alpha}\eta^{\alpha}.$ Similarly in the case when $k=1,$ for $z,w \in \bb M_k$ we write $z:w = \sum_{\lvert \alpha\rvert=k} z(e^{\alpha}).w(e^{\alpha}),$ where we take tensor powers of the standard orthonormal basis $\{e_i\}$ for $\bb R^n.$ This defines an inner product and hence an associated norm $\lvert \cdot\rvert$ on $\bb M_k.$ 

We will consider extremals of the integrand
\begin{equation}
  \cal F(w) = \int_{\Omega} F(\nabla^kw(x)) \,\d x,
\end{equation}
where $F \colon \bb M_k \to \bb R$ and $\nabla^ku$ denotes the $k$\th order partial derivatives of $u.$ These satisfy the Euler-Lagrange equation
\begin{equation}
  (-1)^k \nabla^k : F'(\nabla^ku) = 0
\end{equation}
weakly in $\Omega$ in the sense that
\begin{equation}\label{eq:higher_weakformulation}
  \int_{\Omega} F'(\nabla^ku) : \nabla^k \varphi \,\d x = 0
\end{equation}
for all $\varphi \in C^{\infty}_c(\Omega,\bb R^N).$ The minimising case has been studied for instance in \cite{article:GiaquintaModica79,article:Guidorzi00,article:Kronz02}, and also by the author in \cite{article:Irving21a} where similar arguments are employed to what is considered below.

\begin{hyp}\label{hyp:higherorder_autonomous}
  For $n \geq 2,$ $N,k\geq 1,$ let $F \colon \bb M_k \to \bb R$ be a $C^2$ integrand satisfying the natural growth condition
  \begin{equation}
    \lvert F(z)\rvert \leq K(1+\lvert z\rvert)^q
  \end{equation}
  for all $z \in \bb M_k$ with $q \geq 2,$ and the strict Legendre-Hadamard condition
  \begin{equation}
    F''(z_0)(\xi \tensor \eta^k) : (\xi \tensor \eta^k) \geq 0
  \end{equation}
  for all $z_0$ and all $\xi \in \bb R^N,$ $\eta \in \bb R^n,$ with equality if and only if $\xi \tensor \eta^k = 0.$
\end{hyp}

\begin{thm}[Higher order $\BMO$ $\eps$-regularity theorem]\label{thm:higherorder_regularity}
  Suppose $F$ satisfies Hypotheses \ref{hyp:higherorder_autonomous}, $\Omega$ is a bounded $C^{1,\beta}$ domain for some $\beta \in (0,1),$ and $g \in C^{k,\beta}(\overline\Omega,\bb R^N).$ Then for each $\alpha \in (0,\beta)$ and $M>0,$ there is $\eps>0$ and $\widetilde R_0>0$ such that if $x \in \overline\Omega$ and $0<R<\widetilde R_0$ such that if $u \in W^{k,q}_g(\Omega,\bb R^N)$ is $F$-extremal in $\Omega_R(x_0)$ such that $\lvert (\nabla^k u)_{\Omega_R(x_0)}\rvert \leq M$ and
  \begin{equation}
    \seminorm{\nabla^ku}_{\BMO(\Omega_R(x_0))} \leq \eps,
  \end{equation}
  we have $u$ is $C^{k,\alpha}$ in $\overline{\Omega_{R/2}(x_0)}.$
\end{thm}

Similarly as in Section \ref{sec:autonomous_integrand} for each $M>0$ there is $K_M, \lambda_M > 0$ and a non-decreasing continuous and concave function $\omega_M \colon [0,\infty) \to [0,1]$ satisfying $\omega_M(0)=0$ for which the following holds. If for $z_0 \in \bb M_k$ such that $\lvert z_0\rvert\leq M$ we define
\begin{equation}
  F_{z_0}(z) = F(z_0+z) - F(z_0) - F'(z_0)z.
\end{equation}
This satisfies identical growth and perturbation estimates as in (\ref{eq:autonomous_shifted_estimate2}), (\ref{eq:autonomous_F_pertubation}), namely
\begin{align}
  \lvert F_{z_0}(z)\rvert &\leq K_M( \lvert z\rvert^2 + \lvert z\rvert^q), \label{eq:higher_growth1}\\
  \lvert F'_{z_0}(z)\rvert &\leq K_M( \lvert z\rvert + \lvert z\rvert^{q-1}),\label{eq:higher_growth2}\\
  \lvert F_{z_0}''(0)\rvert &\leq K_M, \label{eq:higher_second}\\
  \lvert F_{z_0}''(0)z - F_{z_0}'(z)\rvert &\leq K_M\,\omega_M(\lvert z\rvert)\left(  \lvert z\rvert + \lvert z\rvert^{q-1} \right) \label{eq:higher_pertubation}
\end{align}
for all $z \in \bb M_k$,
along with the coercivity estimate
\begin{equation}\label{eq:higher_coercivity}
  \int_{\bb R^n} F_{z_0}''(0) \nabla^k\varphi : \nabla^k\varphi \,\d x \geq \lambda_M \int_{\bb R^n} \lvert \nabla^k\varphi\rvert^2 \,\d x
\end{equation}
for all $\varphi \in C^{\infty}_c(\bb R^n,\bb R^N).$

We will also need the following extension of \textsc{Campos Cordero}'s estimate (Lemma \ref{lem:boundary_aconstruction}).

\begin{lem}\label{lem:general_camposcordero}
  Suppose $\Omega$ is a bounded $C^{k,\beta}$ domain for some $\beta \in (0,1)$ and $p>\frac32,$ then there is $R_0>0$ such that for all $x_0 \in \partial\Omega,$ $0<R<R_0$ and $v \in W^{k,p}(\Omega_R(x_0),\bb R^N)$ such that $\nabla_{\nu}^jv = \nabla^jv \cdot \nu^j = 0$ on $\partial\Omega \cap B_R(x_0)$ for each $0\leq j \leq k-1,$ we have the estimate
  \begin{equation}
    \begin{split}
      &\left(\dashint_{\Omega_R(x_0)} \lvert \nabla^kv - (\nabla^kv \cdot \nu_{x_0}^k)_{\Omega_R(x_0)} \tensor \nu_{x_0}^k\rvert^p \,\d x \right)^{\frac1p} \\
      &\qquad\leq \left(\dashint_{\Omega_R(x_0)} \lvert \nabla^kv - (\nabla^kv)_{\Omega_R(x_0)}\rvert^p \,\d x \right)^{\frac1p} + C\lvert (\nabla^kv)_{\Omega_R(x_0)}\rvert R^{\beta},
    \end{split}
  \end{equation}
  with $C=C(n,k,\beta,p,\Omega).$
\end{lem}

\begin{proof}
  As in the $k=1$ case, by translation and rotation we can assume that $x_0 = 0$ and $\nu(x_0) = e_n$, and put
  \begin{equation}
    \tilde v(x) = v(x) - (\nabla_n^kv)_{\Omega_R} \frac{\rho(x)^k}{\lvert \nabla\rho(0)\rvert^k}.
  \end{equation} 
  Here $\rho$ is the defining function from Section \ref{sec:boundary_averages}, which can be chosen to be of class $C^{k,\beta}$ since $\partial\Omega$ is of this regularity.

  \textbf{Claim}: For any multi-index $\lvert \tilde\alpha\rvert \leq k-1$ and $1 \leq i \leq n-1,$ there is $C>0$ such that
  \begin{equation}\label{eq:gradientavg_preiterate}
    \begin{split}
    \left\lvert \dashint_{\Omega_R} \nabla^{\tilde\alpha}\nabla_i \tilde v \,\d x \right\rvert 
    &\leq C\left\lvert \dashint_{\Omega_R} \nabla^{\tilde\alpha}\nabla_n \tilde v \,\d x \right\rvert \\
    &\quad + C \left( \dashint_{\Omega_R} \lvert \nabla^k v - (\nabla^k v)_{\Omega_R}\rvert^p \,\d x \right)^{\frac1p} + C\lvert (\nabla^k_nv)_{\Omega_R}\rvert R^{\alpha}.
    \end{split}
  \end{equation} 
  \textit{Proof of claim}: Arguing as in the proof of Lemma \ref{lem:boundary_aconstruction}, applying \eqref{eq:cordero_onederivative} with $\nabla^{\tilde\alpha}\tilde v$ in place of $\tilde v$ gives
  \begin{equation}
    \left\lvert \dashint_{\Omega_R} \nabla^{\tilde\alpha}\nabla_i \tilde v \,\d x \right\rvert \leq C \left( \dashint_{\Omega_R} \lvert \nabla^{\tilde\alpha}\nabla_n\tilde v\rvert^p \,\d x \right)^{\frac1p},
  \end{equation} 
  and so by the triangle inequality
  \begin{equation}
    \left\lvert \dashint_{\Omega_R} \nabla^{\tilde\alpha}\nabla_i \tilde v \,\d x \right\rvert \leq \left\lvert \dashint_{\Omega_R} \nabla^{\tilde\alpha}\nabla_n \tilde v \,\d x \right\rvert + \left( \dashint_{\Omega_R} \lvert \nabla^{\tilde\alpha}\nabla_n \tilde v - (\nabla^{\tilde\alpha}\nabla_n \tilde v)_{\Omega_R}\rvert^p \,\d x \right)^{\frac1p}.
  \end{equation} 
  The second term can be estimated as
  \begin{equation}
    \begin{split}
    \left( \dashint_{\Omega_R} \lvert \nabla^{\tilde\alpha}\nabla_n \tilde v - (\nabla^{\tilde\alpha}\nabla_n \tilde v)_{\Omega_R}\rvert^p \,\d x \right)^{\frac1p}
    &\leq \left( \dashint_{\Omega_R} \lvert \nabla^k v - (\nabla^k v)_{\Omega_R}\rvert^p \,\d x \right)^{\frac1p} \\
      +& \frac{\lvert (\nabla^k_nv)_{\Omega_R}\rvert}{\lvert \nabla \rho(0)\rvert^k}\left( \dashint_{\Omega_R} \lvert \nabla^k (\rho^k) - (\nabla^k (\rho^k))_{\Omega_R}\rvert^p \,\d x \right)^{\frac1p}.
  \end{split}
  \end{equation} 
  To estimate the $\rho^k$ term we use the uniform estimate
  \begin{equation}
    \frac1{\lvert \nabla\rho(0)\rvert^k}\lvert \nabla^k(\rho^k)(x) - \nabla^k(\rho^k)(0)\rvert \leq C R^{\beta}
  \end{equation} 
  holding for all $x \in \Omega_R$, which follows by noting that $\nabla^k(\rho^k)$ is of class $C^{0,\beta}$ such that $\nabla^k(\rho^k)(0) = (\nabla\rho(0))^k$.
  Combining the estimates the claim follows.

  We can now conclude by iterating \eqref{eq:gradientavg_preiterate} to show that for any $\lvert \alpha\rvert = k,$ we have
  \begin{equation}\label{eq:highercampos_final1}
    \lvert (\nabla^{\alpha}\tilde v)_{\Omega_R(x_0)}\rvert \leq C \lvert (\nabla_n^k\tilde v)_{\Omega_R(x_0)}\rvert + C \left( \dashint_{\Omega_R} \lvert \nabla^k v - (\nabla^k v)_{\Omega_R}\rvert^p \,\d x \right)^{\frac1p} + C\lvert (\nabla^k_nv)_{\Omega_R}\rvert R^{\beta}.
  \end{equation} 
  To pass from $v$ to $\tilde v$ we note that
  \begin{equation}\label{eq:highercampos_final2}
    \lvert (\nabla^{\alpha}v)_{\Omega_R(x_0)}\rvert - \lvert (\nabla^{\alpha}\tilde v)_{\Omega_R(x_0)}\rvert \leq \lvert (\nabla_n^k v)_{\Omega_R}\rvert \frac{\lvert (\nabla^{\alpha}(\rho^k))_{\Omega_R}\rvert}{\lvert \nabla\rho(0)\rvert^k} \leq C \lvert (\nabla_n^k v)_{\Omega_R}\rvert R^{\beta},
  \end{equation} 
  noting again that $\nabla^{\alpha}(\rho^k)$ is a $C^{0,\beta}$-function vanishing at the origin.
  Hence we can conclude by estimating
  \begin{equation}
    \begin{split}
      &\left(\dashint_{\Omega_R(x_0)} \lvert \nabla^kv - (\nabla^kv \cdot \nu_{x_0}^k)_{\Omega_R(x_0)} \tensor \nu_{x_0}^k\rvert^p \,\d x \right)^{\frac1p} \\
      &\qquad\leq \left(\dashint_{\Omega_R(x_0)} \lvert \nabla^kv - (\nabla^kv)_{\Omega_R(x_0)}\rvert^p \,\d x \right)^{\frac1p} + \sum_{\substack{\lvert \alpha \rvert = k,\\ \alpha \neq e_n^k}}\lvert (\nabla^{\alpha}v)_{\Omega_R(x_0)}\rvert\\
      &\qquad\leq C\left(\dashint_{\Omega_R(x_0)} \lvert \nabla^kv - (\nabla^kv)_{\Omega_R(x_0)}\rvert^p \,\d x \right)^{\frac1p} + \lvert (\nabla^kv)_{\Omega_R(x_0)}\rvert R^{\beta},
    \end{split}
  \end{equation}
  where we used both \eqref{eq:highercampos_final1} and \eqref{eq:highercampos_final2} in the last line.
\end{proof}

With this technical estimate in hand, we can turn to the proof of Theorem \ref{thm:higherorder_regularity}. We fix $x_0 \in \overline\Omega$ and chose $R>0$ such that either $B_R(x_0) \subset\Omega,$ or $x_0 \in \partial\Omega$ and $R <R_0$ with $R_0$ as in the start of Section \ref{sec:boundaryregularity}. In the interior case we let $a : \bb R^n \to \bb R^N$ the $k$\th order polynomial satisfying
\begin{equation}
  \nabla^ka_R(x) = \frac{n+2}{R^2} \dashint_{B_R(x_0)} \nabla^{k-1} u \tensor (x-x_0) \,\d x
\end{equation}
and $\left(\nabla^j(u-a_R)\right)_{B_R(x_0)}=0$ for each $0 \leq j \leq k-1,$ and in the boundary case we take
\begin{equation}
  a_R(x) = \xi_R \frac{\rho(x)^k}{\lvert \nabla\rho(x_0)\rvert^k} = \frac{((u-g) \rho^k)_{\Omega_R(x_0)}}{(\rho^{2k})_{\Omega_R(x_0)}} \rho(x)^k.
\end{equation}
We then set $w=u-g-a_R$ and $z_R = \nabla^ka_R(x_0) + (\nabla^kg)_{\partial\Omega_R}$ omitting the $g$-terms in the interior case. Since $\rho$ vanishes at $x_0$ we note that $\nabla^k(\rho^k)(x_0) = (\nabla\rho(x_0))^k,$ and so $\nabla^ka_R(x_0) = \xi_R \tensor \nu_{x_0}^k$ in the boundary case. We assume $\lvert (\nabla^k u)_{\Omega_R(x_0)}\rvert \leq M,$ so then $\lvert z_R\rvert \leq \widetilde M = CM.$
As before write $\widetilde F = F_{z_R}$.
In the below we will focus on the boundary case; the interior case is similar but usually simpler.

\noindent\textbf{Step 1: Caccioppoli-type inequality}: We will show that
\begin{equation}\label{eq:higher_caccioppoli}
  \begin{split}
    &\dashint_{\Omega_{R/2}} \lvert \nabla w\rvert^2 \,\d x \\
    &\qquad\leq C\,\gamma\left(\seminorm{\nabla^k u}_{\BMO(\Omega_R)}\right) \dashint_{\Omega_R} \lvert \nabla^ku-(\nabla^ku)_{\Omega_R}\rvert^2 \,\d x + \frac{C}{R^{2k}} \dashint_{\Omega_R} \lvert w\rvert \,\d x + CR^{2\beta},
  \end{split}
\end{equation}
with $\gamma(t) = \min\{1,\omega_{\widetilde M}(t)(1+t^{2(q-2)}) + t^{2(q-2)}\}$, omitting the $t^{2(q-2)}$ terms if $q =2$.

This will involve a slight modification to account for intermediate derivatives.
Fix $0<t<s<R$ and let $\eta \in C^{\infty}_c(B_R(x_0))$ such that $1_{B_t} \leq \eta \leq 1_{B_s}$ with $\lvert \nabla^j\eta\rvert \leq C(s-t)^{-j}$ for each $0\leq j \leq k.$
Then applying the coercivity estimate (\ref{eq:higher_coercivity}) to $\eta^k w$ and testing the equation (\ref{eq:higher_weakformulation}) against $\eta^{2k}w$ we have 
\begin{equation}
  \begin{split}
      \lambda_{\widetilde M} &\dashint_{\Omega_R} \lvert \nabla^k(\eta^k w)\rvert^2 \,d x \\
                             &\leq \dashint_{\Omega_R} \widetilde F''(0)\nabla^k(\eta^k w) : \nabla^k(\eta^k w) \,\d x - \dashint_{\Omega_R} \widetilde F'(\nabla^k u -z_{R}): \nabla^k(\eta^{2k}w) \,\d x \\
      &= \dashint_{\Omega_R} \eta^k\left( \widetilde F''(0)(\nabla^k u - z_R) - \widetilde F'(\nabla^k u - z_R)\right) : \nabla^k(\eta w) \,\d x \\
      &\quad+\dashint_{\Omega_R} \eta^k \widetilde F''(0)(z_R - \nabla^k a_R - \nabla^k g) : \nabla(\eta^k w) \,\d x \\
      &\quad+ \sum_{j=0}^{k-1}\dashint_{\Omega_R}  \widetilde F''(0)\nabla^{k-j} (\eta^k) \tensor \nabla^jw : \nabla(\eta w) \,\d x\\
      &\quad- \sum_{j=0}^{k-1}\dashint_{\Omega_R} \widetilde F'(\nabla^k u -z_R) : \nabla^{k-j}(\eta^k) \tensor \nabla^j(\eta^k w) \,\d x \\
      &= I + I\!I + I\!I\!I + I\!V.
  \end{split}
\end{equation}
The for the last term $I\!V$ we use \eqref{eq:higher_growth1} and uniform bounds on $\eta$ to estimate
\begin{equation}
  \begin{split}
    I\!V &\leq CK_{\widetilde M} \sum_{j=0}^{k-1} \dashint_{\Omega_R}\eta^k\left( \lvert \nabla^ku-z_R\rvert+\lvert\nabla^ku-z_R\rvert^{q-1} \right) \lvert \nabla^{k-j}\eta\rvert \lvert \nabla^jw \rvert \,\d x,\\
         &\leq CK_{\widetilde M}  \sum_{i=0}^{k-1} \frac1{(s-r)^j} \dashint_{\Omega_R} \left( \lvert \eta^k \nabla w\rvert +  \eta^k \lvert\nabla w\rvert^{q-1} \right) \lvert \nabla^j w\rvert \,\d x\\
         &\quad + CK_{\widetilde M}  \sum_{i=0}^{k-1} \frac1{(s-r)^j} \dashint_{\Omega_R} \eta^k\left( \lvert z_R - \nabla^k a_R - \nabla^k g\rvert +  \lvert z_R - \nabla^k a_R - \nabla^k g\rvert^{q-1} \right) \lvert \nabla^j w\rvert \,\d x.
  \end{split}
\end{equation} 
For the remaining terms we estimate $I$ using \eqref{eq:higher_pertubation}, and for $I\!I,$ $I\!I\!I$ we use \eqref{eq:higher_second}.
By splitting terms using Young's inequality to absorb terms of the form $\nabla^k(\eta^kw)$(as in the proof of Lemma \ref{lem:autonomous_boundary_caccioppoli}) we arrive at
\begin{equation}\label{eq:higher_precaccioppoli}
  \begin{split}
    \int_{\Omega_R} \lvert \nabla^k(\eta^k w)\rvert^2 \,\d x &\leq C \int_{\Omega_R} \omega_{\widetilde M}(\lvert \nabla^ku-z_R\rvert) \left( \lvert \nabla^k u -z_R\rvert^2 + \lvert \nabla^k u - z_R\rvert^{2(q-1)}\right) \,\d x  \\
      &\quad + C \int_{\Omega_R} \lvert \nabla^k u - z_R\rvert^{2(q-1)} \,\d x \\
      &\quad + C \int_{\Omega_R} \lvert \nabla^k a_R - \nabla^kg - z_R\rvert^2 + \lvert \nabla^k a_R - \nabla^kg - z_R\rvert^{2(q-1)} \,\d x\\
      &\quad + C \sum_{j=0}^{k-1} \frac1{(s-t)^{2j}}\int_{\Omega_s} \lvert \nabla^{k-j}w\rvert^2 \,\d x.
  \end{split}
\end{equation}
where the second term does not arise if $q = 2$.
For the last term we use the interpolation estimate to bound the intermediate derivatives $\norm{\nabla^{k-j}w}_{L^2(\Omega_s)},$ using for instance in \cite[Lemma 5.6]{book:AdamsFournier03} (applied in $B_s(x_0)$ after extending by zero). Applying this for the terms we can bound
\begin{equation}
  C \sum_{j=0}^{k-1} \frac1{(s-t)^{2j}}\int_{\Omega_s} \lvert \nabla^{k-j}w\rvert^2 \,\d x \leq \frac12 \int_{\Omega_t} \lvert \nabla^kw\rvert^2 \,\d x + \frac{C}{(s-t)^{2k}} \int_{\Omega_s} \lvert w\rvert^2 \,\d x,
\end{equation}
so then we can absorb the $\nabla^k w$ term by a standard iteration argument (for instance \cite[Lemma 6.1]{book:Giusti03}). For the remaining terms we can bound $\lvert \nabla^ka_R-\nabla^kg-z_R\rvert \leq CR^{\beta}$ and for the $\lvert \nabla^ku-z_R\rvert$ term we note that
\begin{equation}
  \begin{split}
    &\lvert \xi_R - (\nabla^k(u-g) \cdot \nu_{x_0}^k)_{\Omega_R}\rvert \\
    &\quad\leq C\left(\dashint_{\Omega_R} \left\lvert  \nabla^k(u -g) - (\nabla^k(u-g)_{\Omega_R} \cdot \nu_{x_0}^k)\tensor \nu_{x_0}^k\right\rvert^2  \,\d x \right)^{\frac12} + CR^{\beta} \\
    &\quad\leq C \left(\dashint_{\Omega_R} \left\lvert  \nabla^k(u -g) - (\nabla^k(u-g))_{\Omega_R}\right\rvert^2  \,\d x \right)^{\frac12} + CR^{\beta}.
  \end{split}
\end{equation}
Here the first inequality generalises the estimate of \textsc{Kronz} \cite{article:Kronz05} used in Remark \ref{rem:affine_subtraction}, and involves noting that $\dashint_{\Omega_R} \rho^{2k} \lvert \nabla\rho(x_0)\rvert^{-2k} \,\d x \sim R^{2k}$ for $R>0$ sufficiently small and applying the Poincar\'e inequality $k$-times. In the second line we apply Lemma \ref{lem:general_camposcordero}. 
Now we can replace $z_R$ with $(\nabla u)_{\Omega_R}$ in the first two terms in \eqref{eq:higher_precaccioppoli}, allowing us to apply the modular Fefferman-Stein estimate (Corollary \ref{cor:BMO_modulus}) and the John-Nirenberg inequality (Proposition \ref{prop:global_JohnNirenberg}) to infer the claimed estimate \eqref{eq:higher_caccioppoli}.

\noindent\textbf{Step 2: Harmonic approximation}:
Now we take the unique $h \in W^{k,2}(\widetilde\Omega_R,\bb R^N)$ solving the Dirichlet problem
\begin{equation}
  \pdeproblem{(-1)^k\nabla^k : \widetilde F''(0)\nabla^kh}{0}{\widetilde\Omega_R,}{h}{\nabla_{\nu}^jw}{\partial\widetilde\Omega_R \text{ for all } 0 \leq j \leq k-1,}
\end{equation}
where $\widetilde\Omega_R$ is as in Proposition \ref{lem:boundary_linearelliptic}, noting it can be chosen to be $C^{k,\beta}$ to match the regularity of the boundary.
For the duality argument we also consider the unique $\phi \in W^{k,2^*}(\widetilde\Omega_R,\bb R^N)$ to
\begin{equation}
  \pdeproblem{(-1)^k\nabla^k : \widetilde F''(0)\nabla^k\phi}{\nabla^{k-1}(w-h)}{\widetilde\Omega_R,}{\nabla_{\nu}^j\phi}{0}{\partial\widetilde\Omega_R \text{ for all } 0 \leq j \leq k-1,}
\end{equation}
which we claim satisfies the scaled estimate
\begin{equation}
  \norm{\nabla^k\phi}_{L^{2^*}(\widetilde\Omega_R)} \leq R^{-1} \norm{\nabla^{k-1}(w-h)}_{L^2(\widetilde\Omega_R)}.
\end{equation}
For the excess decay estimate we will also need the H\"older estimate
\begin{equation}
  \seminorm{\nabla^kh}_{C^{0,\beta}(\overline{\Omega_{R/2}})} \leq C \dashint_{\Omega_R} \lvert \nabla^kh\rvert^2 \,\d x.
\end{equation}
These results go back to \cite{article:Campanato67} (see also \cite{article:ADN2}), but they can also be straightforwardly adapted from the second order case detailed in \cite[Chapter 10]{book:Giusti03}.

Given these estimates we can argue analogously to the proofs of Lemmas \ref{lem:autonomous_harmonic_approximation}, \ref{lem:autonomous_harmonic_boundary} to show that
\begin{equation}\label{eq:claim_harmonicapprox}
  \frac1{R^{2}}\dashint_{\widetilde\Omega_R} \lvert \nabla^{k-1}(w-h)\rvert^2 \,\d x \leq  C\, \gamma\left(\seminorm{\nabla^ku}_{\BMO(\Omega_R)}\right) \dashint_{\Omega_R} \lvert \nabla^ku-(\nabla^ku)_{\Omega_R}\rvert^2 \,\d x + CR^{2\beta},
\end{equation}
with $\gamma(t) = \min\{1,\omega_{\widetilde M}(t)^{\frac1n}(1+ t^{2(q-2)})\},$ suitably modified if $n=2.$
Indeed we can write
\begin{equation}
  \begin{split}
    &\dashint_{\widetilde\Omega_R} \lvert \nabla^{k-1}(w - h)\rvert^2 \,\d x\\
    &\qquad= \dashint_{\widetilde\Omega_R} \widetilde F''(0)(\nabla^kw - \nabla^kh) : \nabla^k\phi - \widetilde F'(\nabla^k u - z_R) : \nabla^k\phi \,\d x \\
    &\qquad\leq K_{\widetilde M} \dashint_{\widetilde\Omega_R} \omega_M(\lvert\nabla^k u - z_R\rvert) \left( \lvert\nabla^k u - z_R\rvert + \lvert\nabla^ku - z_R\rvert^{q-1} \right) \lvert\nabla^k\phi\rvert \,\d x\\
    &\qquad\quad + K_{\widetilde M} \dashint_{\widetilde\Omega_R} \lvert\nabla^k a_R - \nabla^k g - z_R\rvert \lvert\nabla^k\phi\rvert \,\d x,
  \end{split}
\end{equation} 
and we split the first term using H\"older, invoking the $L^{2^*}$ estimates for $\nabla^k\phi.$ Replacing $z_R$ by $(\nabla u)_{\Omega_R}$ and sing the John-Nirenberg inequality, the claimed estimate \eqref{eq:claim_harmonicapprox} follows.

\noindent\textbf{Step 3: Excess decay estimate}:
Finally to conclude we consider the higher-order excess
\begin{equation}
  E(x,r) = \dashint_{\Omega_r(x)} \lvert \nabla^ku-(\nabla^ku)_{\Omega_r(x)}\rvert^2 \,\d y.
\end{equation}
Then assuming $\lvert (\nabla^ku)_{\Omega_{2\sigma r}(x)}\rvert, \lvert (\nabla^ku)_{\Omega_r(x)}\rvert \leq 2^{3n+1}M$ we can combine the previous two estimates to deduce the decay estimate
\begin{equation}
  E(x,\sigma r) \leq C \left(\sigma^{2\beta} + \sigma^{-(n+2k)} \gamma\left(\seminorm{\nabla^ku}\right)_{\BMO(\Omega_r(x))}\right)E(x,r) + C\sigma^{-(n+2k)}r^{2\beta}.
\end{equation}
Now we can iterate in the usual way to establish Theorem \ref{thm:higherorder_regularity}.

\subsection*{Acknowledgements}
The author would like to thank Jan Kristensen for the many helpful discussions and suggestions.
The author is also thankful to the anonymous referee for carefully reading the original manuscript and for proving valuable feedback; in particular for pointing out an error in proof of Lemma \ref{lem:general_camposcordero} which has since been amended.



\end{document}